\documentclass[11pt,twoside,reqno]{amsart}

\usepackage[OT1]{fontenc}
\usepackage[latin1]{inputenc}
\usepackage{type1cm}
\usepackage{amssymb}
\usepackage[left=2.3cm,top=3cm,right=2.3cm]{geometry}        
\usepackage{version}         
\usepackage[active]{srcltx}  
\usepackage[colorlinks = true, citecolor = black, linkcolor = black, urlcolor = black, pdfstartview=FitH]{hyperref}

\let\oldtocsection=\tocsection
\let\oldtocsubsection=\tocsubsection
\let\oldtocsubsubsection=\tocsubsubsection
\renewcommand{\tocsection}[2]{\hspace{0em}\textbf{\oldtocsection{#1}{#2}}}
\renewcommand{\tocsubsection}[2]{\hspace{1.8em}\oldtocsubsection{#1}{#2}}
\renewcommand{\tocsubsubsection}[2]{\hspace{3em}\oldtocsubsubsection{#1}{#2}}

\numberwithin{equation}{section}

\theoremstyle{plain}
\newtheorem{theorem}{Theorem}[section]
\newtheorem{corollary}[theorem]{Corollary}
\newtheorem{proposition}[theorem]{Proposition}
\newtheorem{lemma}[theorem]{Lemma}

\theoremstyle{remark}
\newtheorem{remark}[theorem]{Remark}

\theoremstyle{definition}
\newtheorem{definition}[theorem]{Definition}

\newtheorem*{construction}{Construction}
\newtheorem*{notation*}{Notation}
\newtheorem*{choice}{Choice of $n_0$}

\newcommand{\R}{\mathbb{R}}
\newcommand{\C}{\mathbb{C}}
\newcommand{\Q}{\mathbb{Q}}
\newcommand{\N}{\mathbb{N}}

\renewcommand{\P}{\mathbb{P}}

\newcommand{\cB}{\mathcal{B}}

\newcommand{\cM}{\mathcal{M}}
\newcommand{\cH}{\mathcal{H}}
\newcommand{\cL}{\mathcal{L}}

\renewcommand{\d}{\mathbf{d}}

\newcommand{\cR}{\mathcal{R}}
\newcommand{\cC}{\mathcal{C}}

\newcommand{\cA}{\mathcal{A}}

\newcommand{\eps}{\varepsilon}

\newcommand{\Z}{\mathbb{Z}}

\renewcommand{\epsilon}{\varepsilon}
\renewcommand{\rho}{\varrho}
\renewcommand{\phi}{\varphi}

\renewcommand{\mod}{\,\,\mathrm{mod\,}}

\renewcommand{\a}{\mathbf{a}}
\renewcommand{\b}{\mathbf{b}}
\renewcommand{\c}{\mathbf{c}}

\renewcommand{\hat}{\widehat}

  \newcommand{\stack}[2]{\genfrac{}{}{0pt}{}{#1}{#2}}

\usepackage[usenames,dvipsnames]{xcolor}

\title{Fourier transforms of Gibbs measures for the Gauss map}

\author[Thomas Jordan]{Thomas Jordan}
\address{Department of Mathematics, University of Bristol, University Walk, Clifton, Bristol, BS8 1TW, England}
\email{thomas.jordan@bristol.ac.uk}

\author[Tuomas Sahlsten]{Tuomas Sahlsten$^{*}$}
\address{Einstein Institute of Mathematics, The Hebrew University of Jerusalem, Givat Ram, Jerusalem 91904, Israel}
\email{tsahlsten@math.huji.ac.il}

\thanks{$^{*}$Corresponding author. T.S. acknowledges the partial support from University of Bristol, the Finnish Centre of Excellence in Analysis and Dynamics Research, Emil Aaltonen Foundation and the European Union (ERC grant $\sharp$306494)}


\begin{document}

\begin{abstract} 
We investigate under which conditions a given invariant measure $\mu$ for the dynamical system defined by the Gauss map $x \mapsto 1/x \mod 1$ is a Rajchman measure with polynomially decaying Fourier transform 
$$|\widehat{\mu}(\xi)| = O(|\xi|^{-\eta}), \quad \text{as }  |\xi| \to \infty.$$ 
We show that this property holds for any Gibbs measure $\mu$ of Hausdorff dimension greater than $1/2$ with a natural large deviation assumption on the Gibbs potential. In particular, we obtain the result for the Hausdorff measure and all Gibbs measures of dimension greater than $1/2$ on badly approximable numbers, which extends the constructions of Kaufman and Queff\'elec-Ramar\'e. Our main result implies that the Fourier-Stieltjes coefficients of the Minkowski's question mark function decay to $0$ polynomially answering a question of Salem from 1943. As an application of the Davenport-Erd\H{o}s-LeVeque criterion we obtain an equidistribution theorem for Gibbs measures, which extends in part a recent result by Hochman-Shmerkin. Our proofs are based on exploiting the nonlinear and number theoretic nature of the Gauss map and large deviation theory for Hausdorff dimension and Lyapunov exponents. 
\end{abstract}

\maketitle

\section{Introduction and the main results}

\subsection{Rajchman measures} 

Given a Borel probability measure $\mu$ on the unit circle $\mathbb{T}$, we define the \textit{Fourier transform} of $\mu$ at the frequency $\xi \in \R$ by the quantity
$$\widehat{\mu}(\xi) = \int e^{-2\pi i \xi x}\, d \mu(x).$$
It is a central question in Fourier analysis and its applications to understand how the behaviour or decay of $\widehat{\mu}$ affect for example the absolute continuity, singularity, geometric or arithmetic structure of the measure $\mu$. If $\widehat{\mu} \to 0$ at infinity, then $\mu$ is called a \textit{Rajchman measure}. Rajchman measures have their root and motivation in the study of sets of uniqueness and multiplicity of Fourier series; see for example \cite{Kaufman1992} and the survey of Lyons \cite{Lyons1995} for a more detailed historical background.

When we know the \textit{rate} of decay for the Fourier transform of a Rajchman measure, we can recover a lot of information from the structure of the measure. For example, a classical bound coming from potential theoretic methods (see for example \cite{Mattila1995} and references therein) shows that if the Fourier transform $\widehat{\mu} \to 0$ \textit{polynomially}, that is, $|\widehat{\mu}(\xi)| = O(|\xi|^{-\eta})$ as $|\xi| \to \infty$ for some $\eta > 0$, then the Hausdorff dimension of the measure $\mu$ satisfies $\dim \mu \geq 2\eta$. This gives rise to the notion of \textit{Salem sets}, that is, a set $A \subset \R$ with Hausdorff dimension $s = \dim A$ that supports a Rajchman measure $\mu$ with Fourier transform $|\widehat{\mu}(\xi)| = O(|\xi|^{-\eta})$ as $|\xi| \to \infty$ with $\eta$ arbitrarily close to $s/2$. Thanks to  Plancherel's  theorem, if the Fourier transform $\widehat{\mu}\in L^2$, then $\mu$ must be absolutely continuous with $L^2$ density, and moreover, if $\widehat{\mu} \in L^1$, then $\mu$ is a continuous function. With the help of the convolution formula $\widehat{\mu \ast \nu} = \widehat{\mu} \cdot \widehat{\nu}$ this kind of powerful information can be linked to the structure and size of the sumsets and distance sets of the supports of Rajchman measures with rapidly enough decaying Fourier transform; for a detailed background, see for example the survey of Mattila \cite{Mattila2004}.

Rajchman measures with polynomially decaying Fourier transforms share many similar properties with Lebesgue measure. The \textit{Erd\H{o}s-Tur\'an inequality} yields a universal bound for the measure $\mu(I)$ of an interval $I$ compared to its length $|I|$ via the Fourier coefficients $\widehat{\mu}(k)$. Lebesgue measure also has the following characteristic property that almost every number is normal and a similar criterion can be deduced for Rajchman measures with polynomial decay as well. Recall that a given a sequence $x_1,x_2,\dots,$ of non-negative real numbers \textit{equidistributes modulo $1$}, if for any interval $I \subset [0,1]$ the frequency of $x_k$ hitting $I$ is the length $|I|$. In other words, there exists
$$\lim_{N \to \infty} \frac{|\{1 \leq k \leq N : x_k \mod 1 \in I\}|}{N} = |I|.$$
A particularly important case is the sequence $x_k = n^k x$, $k \in \N$, where $n \in \N$ is a fixed base. If the sequence $(n^k x)$ equidistributes modulo $1$, the number $x$ is called \textit{$n$-normal}. Moreover, a measure $\mu$ on $[0,1]$ is $n$-\textit{normal} if $\mu$ almost every $x$ is $n$-normal. Invoking the famous \textit{Weyl's criterion} for equidistribution with exponential sums, we obtain the following

\begin{theorem}[Davenport-Erd\H{o}s-LeVeque]\label{thm:equidistribution}
Let $\mu$ be a probability measure on $[0,1]$ and $(s_k)_{k \in \N}$ a sequence of natural numbers. If any $p \neq 0$ satisfies
\begin{align}\label{crit}\sum_{N = 1}^\infty \frac{1}{N^3} \sum_{k,m = 1}^N \widehat{\mu}(p(s_k - s_m)) < \infty,\end{align}
then the sequence $(s_k x)_{k \in \N}$ equidistributes modulo $1$ at $\mu$ almost every $x$.

When $\mu$ is a Rajchman measure with $\widehat{\mu} \to 0$ polynomially, then \eqref{crit} holds for every strictly increasing sequence $(s_k)_{k \in \N}$. In particular, $\mu$ is $n$-normal for any $n \in \N$.
\end{theorem}

This was proved by Davenport, Erd\H{o}s and LeVeque in \cite{DavErdLev1963} and for the proof of the statement in this form, see for example \cite[Theorem 7.2]{QueffelecRamare2003}. Moreover, see also \cite[Corollary 7.4]{QueffelecRamare2003} for the deduction for the corollary on Rajchman measures.

Theorem \ref{thm:equidistribution} can be applied to find normal numbers in small sets which still support Rajchman measures with polynomially decaying Fourier coefficients. This approach has been particularly useful in Diophantine approximation of irrational numbers after the seminal works of Kaufman \cite{Kaufman1980,Kaufman1981}. A classical result of Dirichlet says that for any irrational number $x \in [0,1]$ we can find infinitely many rationals $p/q$ with $|x-p/q| \leq q^{-2}$. A number is \textit{well approximable} if the rate of approximation Dirichlet's theorem gives can be improved: for example, if $\alpha \geq 2$, the class of \textit{$\alpha$-well approximable numbers} $W(\alpha)$ consists of those irrational $x$ with $|x-p/q| \leq q^{-\alpha}$ for infinitely many rationals $p/q$. The set $W(\alpha)$ is when $\alpha > 2$ a fractal with Hausdorff dimension $\dim W(\alpha) = 2/\alpha$ by a classical result of Jarn\'ik-Besicovitch. In \cite{Kaufman1981} Kaufman proved that there is a Rajchman measure $\mu$ on $W(\alpha)$ with polynomial decay 
$$|\widehat{\mu}(\xi)| = o(\log |\xi|)|\xi|^{-1/\alpha},$$ 
which shows that $W(\alpha)$ is a Salem set. The papers by Bluhm \cite{Bluhm1998,Bluhm2000} contain a more detailed proof of Kaufman's result and demonstrate that the set of Liouville numbers (numbers in every $W(\alpha)$) also supports a Rajchman measure; note that as the Liouville numbers have Hausdorff dimension $0$ the decay cannot be polynomial.

A polar opposite to well approximability is the collection of \textit{badly approximable numbers}, that is, those irrational $x \in [0,1]$ for which we can find a constant $c(x) > 0$ such that $|x-p/q| \geq c(x)/q^2$ for all rationals $p/q$. All such numbers are exactly those numbers for which the \textit{continued fraction expansion} $a_1(x),a_2(x),\dots,$ is bounded (see Section \ref{sec:dioph}). Thus this gives rise to the sets
$$B_N = \{x \in [0,1] \setminus \Q : a_i(x) \leq N \text{ for all } i \in \N\}$$
whose union is the set of all badly approximable numbers. As $B_N$ can be interpreted as an attractor to a self-conformal iterated function system $\{x \mapsto 1/(x+a) : a = 1,2,\dots,N \}$ (see Section \ref{sec:dioph}), it is possible to estimate the Hausdorff dimension of $B_N$. In \cite{Good1941} it was estimated $\dim B_N \geq \dim B_2 = 0.531..$ for any $N \in \N$ and Jarn\'ik \cite{Jarnik1928} proved that $\dim B_N \nearrow 1$ as $N \to \infty$. These sets support Rajchman measures:

 \begin{theorem}[Kaufman \& Queff\'elec-Ramar\'e]\label{thm:kaufman}
 Given $N \geq 2$, there exists a Rajchman measure $\mu$ on $B_N$ such that $\widehat{\mu} \to 0$ polynomially.
 \end{theorem}

The measure constructed in this proof is now widely known as the \textit{Kaufman measure}. Theorem \ref{thm:kaufman} was proved by Kaufman \cite{Kaufman1980} for $N \geq 3$ and Queff\'elec-Ramar\'e \cite{QueffelecRamare2003} later extended this result for $N \geq 2$ by modifying the proof of Kaufman and with more careful analysis of continuants. Thus this yields by Theorem \ref{thm:equidistribution} that there are infinitely many badly approximable $n$-normal numbers for any $N \geq 2$ and $n \in \N$. For $N \geq 3$ this was already settled by Baker, who pointed out this could be deduced from Kaufman's result (see for example the lecture notes by Montgomery \cite{Montgomery1994}), but for $B_2$ the Kaufman measure constructed by Queff\'elec-Ramar\'e finally settled this problem for all $N \geq 2$.

Kaufman measures and their construction have since become useful tool in Diophantine approximation. See for example the paper by Pollington and Velani \cite{PollingtonVelani2000} with a connection to the Littlewood conjecture.
 
\subsection{Fourier transforms of invariant measures} 

The motivation for this paper is to investigate under which conditions on a dynamical system a given invariant measure $\mu$ is a Rajchman measure with a polynomial decaying Fourier transform. It turns out the construction of the Kaufman measure on badly approximable numbers in Theorem \ref{thm:kaufman} provides a good reference to understand this problem. 

Kaufman's original proof relied on constructing a measure $\mu$ with large enough dimension supported on $B_N$ which satisfies suitable $\mu$ independence assumption for the measures of cylinders. The independence in the construction allowed the use of law of large numbers to obtain $\mu$ generic growth rates for the \textit{continuants} of the continued fraction expansions and that the generic continuants do not \textit{deviate} much from this generic growth. This controlled growth of generic continuants is crucial in the proof when studying the oscillations of $\mu$ with $|\widehat{\mu}(\xi)|$ for large frequencies $\xi$.

The growth rate of continuants can be explained by the evolution of the dynamical system $([0,1] \setminus \Q,T)$ on irrationals, where $T : [0,1] \setminus \Q \to [0,1]\setminus \Q$ is the  \textit{Gauss map}, defined by
$$T(x) = \frac{1}{x} \mod 1.$$ 
The Gauss map acts as a shift for the continued fraction expansions of irrational numbers and the growth rate of the derivative $(T^n)'(x)$ as $n \to \infty$ of the composition $T^n$ is comparable to the growth of the $n$th continuant of an irrational $x$ (see the end of Section \ref{sec:dioph}). Thus implicitly Kaufman's construction already relies heavily on the dynamics of the Gauss map, which leads to the problem; is the decay of Fourier coefficients in general just a generic property of the dynamical system?

We show that Rajchman measures are indeed quite common for the Gauss map and the classical \textit{Gibbs condition} is enough (see Section \ref{sec:thermo} for a definition). Gibbs measures arise in thermodynamical formalism as equilibrium states for some potential functions and they enjoy a weak form of independence; for example, they are always ergodic $T$ invariant measures. All Bernoulli measures on badly approximable numbers are Gibbs measures and there is a Gibbs measure which is equivalent to the Hausdorff measure $\cH^s$ of dimension $s = \dim B_N$ when restricted to $B_N$. Gibbs measures tend to satisfy strong statistical properties, in particular in this setting we can show they satisfy large deviation results on the generic growth of the Gauss map, which enable us to adapt Kaufman's approach.

However, as the alphabet generated by the Gauss map is infinite, the potential associated to the Gibbs measure $\mu$ may not in general satisfy the large deviation principle, so we need to impose a weak condition on the decay in the tail of the Gibbs measure. More precisely, that there exists $\delta>0$ such that when $n \to \infty$ the probability
\begin{equation}\label{pt}
\mu(\{x \in [0,1] \setminus \Q :a_1(x)\geq n\})=O(n^{-\delta})
\end{equation}
This condition on the distribution of the first continued fraction digit of $x$ is equivalent to the Gibbs measure $\mu$ satisfying the large deviation theory bounds for the Kaufman type arguments we use in the proof; see Proposition \ref{1term} and Remark \ref{rmk:large} below.
For a finite alphabet, which would for example cover the situation of $B_N$ as symbolically it is $\{1,2,\dots,N\}^\N$, such an assumption is unnecessary so the only assumption on the Gibbs measure required is that the dimension is greater than $\frac{1}{2}$. More generally, we can just consider any finite $\cA \subset \N$ and define an irrational $x$ to be $\cA$-\textit{badly approximable}, if the digits of the continued fraction expansion of $x$ are only in $\cA$, and write $x \in B(\cA)$. Then $B_N = B(\{1,2,\dots,N\})$.

\begin{theorem}\label{thm:main} We have the following properties:
\begin{itemize}
\item[(1)] If $\mu$ is any Gibbs measure for the Gauss map which satisfies \eqref{pt} and Hausdorff dimension $\dim \mu> 1/2$, then $\mu$ is a Rajchman measure with Fourier transform
\begin{align}\label{fourierdecay}\widehat{\mu}(\xi)=O(|\xi|^{-\eta}), \quad |\xi| \to \infty,\end{align}
for some $\eta > 0$.
\item[(2)] If $\cA \subset \N$ is finite and $\mu$ is any Gibbs measure for the Gauss map restricted to $B(\cA)$ with dimension $\dim \mu>1/2$, then \eqref{fourierdecay} holds.
\end{itemize}
\end{theorem}

It is worth pointing out the contrast between this result and the case for the map $\times n$ map  $x \mapsto nx\mod 1$ on $\mathbb{T}$, where a measure $\mu$ is $\times n$ invariant if and only if $\widehat{\mu}(nk) = \widehat{\mu}(k)$ for all $k \in \Z$ so the only $\times n$ invariant measure with whose Fourier transform has a power decay is Lebesgue measure. 

After a slight adjustment of the proof of Theorem \ref{thm:main}, it also applies to Hausdorff measures $\cH^s$ on the badly approximable numbers $B(\cA)$ with large enough dimension:

\begin{corollary}\label{hm}
If $\cA \subset \N$ is finite $s = \dim B(\cA) > 1/2$, then the Hausdorff measure $\cH^s$ restricted to $B(\cA)$ is a Rajchman measure with polynomially decaying Fourier transform. In particular, the Hausdorff measure on $B_N$ satisfies this property for any $N \geq 2$.
\end{corollary}

This is due to the fact that Hausdorff measure $\cH^s$ is an $s$-\textit{conformal measure} in our setting and equivalent to a Gibbs measure with uniformly positive and finite densities; see Section 6.7 for more details and references.

Another application of Theorem \ref{thm:main} concerns the Fourier-Stieltjes coefficients of singular monotonic functions $f : [0,1] \to [0,1]$, that is, monotonic $f$ with zero derivative Lebesgue almost everywhere. Such functions arise naturally from singular measures. If $\mu$ is a singular probability measure with respect to Lebesgue measure, with support $[0,1]$ and without atoms, then the cumulative distribution function $x \mapsto \mu([0,x])$ is an increasing singular monotonic function. Moreover, if we have an increasing singular monotonic function $f : [0,1] \to [0,1]$, then defining a Stieltjes measure $\mu$ by the formula $\mu([x,y]) = f(y) - f(x)$ for all reals $0 \leq x < y \leq 1$ has $f$ as the cumulative distribution function. The \textit{Fourier-Stieltjes coefficients} of $f$ are then precisely the Fourier coefficients $\widehat{\mu}(k)$, $k \in \Z$, of the measure $\mu$. There has been much research to understand the asymptotics of these coefficients for various singular functions, see Salem's paper \cite{Salem1943} for references. 

All Gibbs measures $\mu$ for the Gauss map that are fully supported and not equal to the \textit{Gauss measure} (see Remark \ref{rmk:examples}(2) for a definition) give rise to a singular monotonic function by the cumulative distribution function $x \mapsto \mu([0,x])$. Indeed, any ergodic measure for the Gauss map not equal to the \textit{Gauss measure} (see Remark \ref{rmk:examples}(2)) has to be singular with respect to Gauss measure and hence to Lebesgue measure. Thus if a singular monotonic function $f$ gives rise to a Gibbs measure $\mu$ for a Gauss map with the requirement \eqref{pt} is on the tail of $\mu$ and $\dim \mu > \frac{1}{2}$ satisfied, we can apply Theorem \ref{thm:main} to have that the Fourier-Stieltjes coefficients of $f$ decay to $0$ at infinity with a polynomial rate. An example of this application is given in the following.

A famous construction of a singular monotonic function given by Minkowski in \cite{Minkowski1911} is the \textit{Minkowki's question mark function}, which is defined for irrationals $x \in [0,1]$ by
$$?(x) = 2\sum_{n = 1}^\infty \frac{(-1)^{n+1}}{2^{a_1(x) + \dots + a_n(x)}}$$
and for rationals $x \in [0,1]$ with continued fraction expansion $[a_1(x),\dots,a_m(x)]$ by replacing the sum just up to $m$ and setting $?(0) = 0$. The original motivation for the construction of $?$ was to relate the continued fraction expansions of quadratic irrational numbers to their binary expansions. It turns out that $?$ has many fractal features, see for example the work by Kesseb{\"o}hmer and Stratmann \cite{KesseStrat2008} on multifractal analysis and the references therein. Let $\mu$ is the Stieltjes measure defined by $?$. It dates back to the work of Salem \cite{Salem1943} in 1943 where he proved that the Fourier coefficients of $\mu$ converge 'in average' to $0$, that is,
$$\frac{1}{2n+1}\sum_{k = -n}^n |\widehat{\mu}(k)| = O(n^{-\eta})$$
as $n \to \infty$ for a suitable $\eta > 0$. In the same paper Salem asked can this be strengthened to $\widehat\mu \to 0$, i.e. whether $\mu$ is Rajchman or not. This problem has been re-visited over the years, for example it was recently mentioned in \cite{CHM2014}. 

We see that $\mu$ is indeed a Gibbs measure for the Gauss map; in particular, it is a \textit{Bernoulli measure} (see Remark \ref{rmk:examples}(1) for a definition) associated to the weights $p_a = 2^{-a}$, $a \in \N$. The tail event $\{a_1(x) \geq n\}$ has an exponential decay:
$$\mu(\{x \in [0,1] \setminus \Q :a_1(x)\geq n\}) = 2^{-n+1}$$
so \eqref{pt}  is satisfied. Finally in \cite{Kinney1960} the Hausdorff dimension of $\mu$ is found. It is shown that
$$\dim \mu = \frac{\log 2}{2\int_0^1 \log(1+x) \, d\mu(x)} > \frac{1}{2}.$$
Thus the measure $\mu$ satisfies the assumptions of Theorem \ref{thm:main} and we obtain a corollary, which answers to Salem's original question and also gives a bound for the rate for the decay:

\begin{corollary}\label{cor:salem} The Fourier-Stieltjes coefficients of the Minkowski's question mark function decay to $0$ at infinity with a polynomial rate.
\end{corollary}

This application of Theorem \ref{thm:main} was pointed out by Persson in \cite{persson2015}. Note an alternative approach to solving Salem's problem was presented in \cite{Yakubovich2015}.

Furthermore, the main Theorem \ref{thm:main} together with the equidistribution Theorem \ref{thm:equidistribution} yields that suitable Gibbs measures will give us normal numbers in their support:

\begin{corollary}\label{cor:gibbsnormal}
\begin{enumerate}
\item
If $\mu$ is any Gibbs measure for the Gauss map which satisfies \eqref{pt} and Hausdorff dimension $\dim \mu> 1/2$, then $(s_k x)_{k\in \N}$ equidistributes modulo $1$ at $\mu$ almost every $x$ for any strictly increasing sequence of natural numbers $(s_k)_{k \in \N}$. In particular, $\mu$ is $n$-normal for any $n \in \N$.
\item
If $\cA \subset \N$ is finite and $\mu$ is any Gibbs measure for the Gauss map restricted to $B(\cA)$ with dimension $\dim \mu>1/2$, then the same conclusion holds for $\mu$.
\end{enumerate}
\end{corollary}

Recently, in \cite{HochmanShmerkin2013} Hochman and Shmerkin presented a general criterion for the $n$-normality of a measure $\mu$, which adapts to many general nonlinear iterated function systems including the one determined by the Gauss map. In particular, they obtained the following in \cite[Theorem 1.12]{HochmanShmerkin2013}.

\begin{theorem}[Hochman-Shmerkin]\label{thm:hochmanshmerkin}
Any Gibbs measure for the Gauss map supported on $B(\cA)$ for a finite $\cA \subset \N$ is $n$-normal for any $n \in \N$.
\end{theorem}

In their proof Hochman and Shmerkin did not rely on large deviations, Fourier transforms or the Davenport-Erd\H{o}s-LeVeque criterion on equidistribution but rather they used a completely different approach by studying the ergodic properties of the so called \textit{scenery flow} invariant distributions generated by the Gibbs measure. However, the results of Hochman and Shmerkin do not apply to give the more general statement on equidistribution of $(s_k x)_{k\in\N}$ for \textit{any} strictly increasing sequence $(s_k)_{k\in\N}$ obtained from Davenport-Erd\H{o}s-LeVeque, so our result extends their results in this respect even in the case of finite alphabets with the caveat of the curious Hausdorff dimension assumption. The $n$-normality part of Corollary \ref{cor:gibbsnormal}(1) was conjectured also in \cite{HochmanShmerkin2013}: ``\textit{...It seems very likely that the result holds also for Gibbs measures when $\cA \subset \N$ is infinite, under standard assumptions on the Gibbs potential...}'', so dropping the dimension assumption $\dim \mu > 1/2$ in Corollary \ref{cor:gibbsnormal}(1) should be possible at least for the $n$-normality.

We finish this section by outlining how the rest of the paper is structured. In Sections \ref{sec:dioph} and \ref{sec:thermo} we describe the basic results from the theory of continued fractions and thermodynamic formalism that we use. In Section \ref{sec:largedeviations} we prove the large deviation results we need and in Section \ref{sec:decomposition} we describe how we can use these results to decompose the Gibbs measure, where we can apply Kaufman type arguments. We can then complete the proof of Theorem \ref{thm:main} in Section \ref{sec:proof} and we finish the paper by commenting on some possible extensions of our work beyond the system $([0,1]\setminus \Q,T)$ to more general nonlinear dynamical systems.

\section{Diophantine approximation and the Gauss map} 
\label{sec:dioph}

During the course of the proof of Theorem \ref{thm:main} we will use heavily the language of continued fractions. In this section we will introduce our notation for continued fractions and state the elementary properties of continued fraction that we make use of. For the majority of the results presented here, we refer to the book by Khinchin \cite{Khinchin1997}. Write $\N = \{1,2,\dots\}$. Given digits $a_1,\dots,a_{n} \in \N$, where $n \in \N$, a \textit{continued fraction} is the rational number
$$[a_1,a_2,\dots,a_{n}] = \cfrac{1}{a_1 + \cfrac{1}{a_2 + \cfrac{1}{ \ddots + \cfrac{1}{a_{n}}}}}.$$
For each irrational $x \in [0,1]$ we can find unique numbers $a_i(x) \in \N$ such that
$$x = [a_1(x),a_2(x),\dots] := \lim_{n \to \infty} [a_1(x),a_2(x),\dots, a_{n}(x)].$$
This expression is the \textit{continued fraction expansion} of $x$. Thus we obtain a coding of $[0,1] \setminus \Q$ via the words $\a = (a_1,a_2,\dots) \in \N^\N$. For notational purposes, let $\N^n$ and $\N^*$ be the collections of length $n$ and any finite length words of natural numbers, and in the proceedings $\a$ can mean either finite word $(a_1,a_2,\dots,a_n)$ or an infinite word $(a_1,a_2,\dots)$ depending on the context.

\begin{definition}[Continuants] For any finite word $\a = (a_1,a_2,\dots,a_n) \in \N^n$, there exists integers $p_n(\a)$ and $q_n(\a)$ such that 
$$[a_1,a_2,\dots,a_n] = \frac{p_{n}(\a)}{q_n(\a)}.$$
The denominator here $q_n(\a)$ called the \textit{continuant} of the continued fraction $[\a]$. We define the numbers $q_k(\a)$ also for $k < n$ even if the word $\a$ has length bigger than $k$. Then $q_k(\a)$ is just the continuant $q_k(\a|_k)$, where $\a|_k$ is the restriction $(a_1,\dots,a_k)$.
\end{definition}

Continuants satisfy by construction the following important relations:
\begin{remark}
\label{rmk:continuantlemma}
Given $\a \in \N^n$, we have
\begin{itemize}\item[(1)] the recurrence relation
\begin{align}\label{algorithm}q_{n}(\a) = a_{n} q_{n-1}(\a) + q_{n-2}(\a);\end{align}
\item[(2)] relation between numerators and denominators
\begin{align}\label{relation}q_n(\a) p_{n-1}(\a)  - q_{n-1}(\a) p_n(\a) = (-1)^n;\end{align}
\item[(3)] invariance and recovery under \emph{mirroring}
\begin{align}\label{reverse}q_{n}(\a) = q_{n}(\a^\leftarrow)\quad \text{and} \quad q_{n-1}(\a) = p_n(\a^{\leftarrow})\end{align}
where $\a^\leftarrow = (a_n,\dots,a_1)$ is the \emph{mirror image} of $\a = (a_1,\dots,a_n) \in \N^n$.
\end{itemize}
\end{remark}

Continuants can be used to estimate the rate of convergence of continued fractions to a given irrational number. In particular, if $x \in [0,1] \setminus \Q$ and $\a \in \N^n$ is its continued fraction digits up to $n$, then \textit{Dirichlet's theorem} tells us
$$\Big|x-\frac{p_n(\a)}{q_n(\a)}\Big| \leq \frac{1}{q_n(\a)^2}.$$
Continuants themselves have already a lower bound for their growth: in particular, if
$$\theta = \frac{\sqrt 5 + 1}{2}$$
is the \textit{Golden ratio}, then we have the following
\begin{lemma}\label{lma:goldengrowth}
If $\a \in \N^n$, then  $q_n(\a) \geq c_0 \theta^n$ for some constant $c_0 > 0$ independent of $\a$.
\end{lemma}
\begin{proof}
We have that each $a_i \geq 1$ so we have by the recurrence relation \eqref{algorithm}, that the word $1^n = (1,1,\dots,1) \in \N^n$ satisfies $q_{n}(\a) \geq q_n(1^n)$. The recurrence relation \eqref{algorithm} gives the relation 
$$q_n(1^n) = q_{n-1}(1^n) + q_{n-2}(1^n)$$ 
with $q_1(1^1) = 1$ and for convention $q_0(1^0) := 0$, so the sequence $(q_n(1^n))_{n \in \N}$ is the Fibonacci sequence. By Binet's formula this means that it has exponential growth rate given by $c_0 \theta^n$.
\end{proof}

Another way to understand the continuants $q_n$ when $n$ is large, is to interpret the problem with the growth rate of orbits for a suitable dynamical system. This gives rise to the notion of Gauss map:

\begin{definition}[Gauss map and inverse branches] Let $T : [0,1]\backslash\Q \to [0,1]\backslash\Q$ be the \textit{Gauss map} defined by
$$T(x) = \frac{1}{x} \mod 1.$$
Moreover, let $T_a : [0,1] \setminus \Q  \to I_a := [\frac{1}{a+1},\frac{1}{a}] \setminus \Q$ be the inverse branch $(T|_{I_a})^{-1}$ of the Gauss map given by 
$$T_a(x) = \frac{1}{x+a}, \quad x \in [0,1] \setminus \Q.$$
\end{definition}

The Gauss map acts as a shift for the continued fraction digits of $x \in [0,1] \setminus \Q$:
$$T(x) = [a_2(x),a_3(x),\dots] \in [0,1] \setminus \Q.$$
In the language of dynamical systems this means that the Gauss map is conjugate to the full shift $\sigma : \N^\N \to \N^\N$.

The growth rate of the continuants under compositions of the inverse branches 
$$T_\a := T_{a_1} \circ T_{a_2} \circ \dots \circ T_{a_n}$$
for $\a \in \N^n$ is about of the order $q_n(\a)^{-2}$ as the \textit{construction intervals} 
$$I_\a := T_\a([0,1] \setminus \Q) = \{x\in [0,1]\backslash\Q: a_i(x)=a_i\text{ for }1\leq i\leq n\}$$
are of the size $q_n(\a)^{-2}$. More precisely

\begin{lemma} \label{lma:growth}Given $\a \in \N^n$, then we have
\begin{align}\label{sizeofderivative}\frac{1}{4}q_n(\a)^{-2} \leq |T_\a'| \leq q_n(\a)^{-2}\end{align}
In particular, the same bounds hold for the length $|I_\a|$.
\end{lemma}

\begin{proof} By construction the inverse branch $T_\a$ at $x \in [0,1]$ and $\a \in \N^n$ is precisely
\begin{align}\label{inverse}T_\a(x) = [a_1,a_2,\dots,a_n+x] = \frac{p_{n-1}(\a)x+p_n(\a)}{q_{n-1}(\a)x+q_n(\a)} = \frac{p_{n}(\a)}{q_n(\a)} + \frac{(-1)^n x}{(q_{n-1}(\a)x+q_n(\a))q_n(\a)}.\end{align}
This yields that the derivative has the following representation
\begin{align}\label{derivative}T_\a'(x) = \frac{(-1)^n}{(q_{n-1}(\a)x+q_n(\a))^2}.\end{align}
By the recurrence relation \eqref{algorithm} we have $q_{n-1}(\a) \leq q_n(\a)$ so
$$\frac{1}{4q_n(\a)^2} \leq |T_\a'(x)| \leq \frac{1}{q_n(\a)^2}.$$
By computing $T_\a$ at the end points or using mean value theorem we obtain the same bounds for the length $|I_\a|$.
\end{proof}

Using this, we can give a useful quasi-independence for the continuants:

\begin{lemma} \label{lma:quasiindependence} Given $\a \in \N^n$ and $1 \leq j < n$, we have
$$\frac{1}{2} \leq \frac{q_{n}(\a)}{q_{n-j}(a_1,\dots,a_{n-j}) q_j(a_{n-j+1},\dots,a_{n})} \leq 4.$$
\end{lemma}

\begin{proof}
Write $\b = a_1,\dots,a_{n-j}$ and $\c = a_{n-j+1},\dots,a_{n}$. Given $x \in [0,1]$ we have by the chain rule
$$T_\a'(x) = T_\b'(T_\c(x)) T_\c'(x).$$
Then by Lemma \ref{lma:growth}, we have
$$\frac{1}{16q_{n-j}(\b)^2q_{j}(\c)^2} \leq |T_\b'(T_\c(x))| |T_\c'(x)| \leq \frac{1}{q_{n-j}(\b)^2q_{j}(\c)^2}$$
so applying Lemma \ref{lma:growth} again for $T_\a(x)$ and taking a root we have
$$\tfrac{1}{2}q_{n-j}(\b)q_{j}(\c) \leq  q_{n}(\a) \leq 4q_{n-j}(\b)q_{j}(\c).$$
\end{proof}

\section{Thermodynamical formalism and Gibbs measures} 
\label{sec:thermo}

The fact that the Gauss map is conjugate to the full shift on $\N^\N$ enables us to use the thermodynamical formalism for countable shifts developed by Mauldin and Urba\'nski in \cite{MauldinUrbanski2003} and Sarig in the papers \cite{sarig2001,sarig2003}. This also relates to earlier work of Walters in \cite{walters78}. Let us now recall some classical notation and definitions related to thermodynamical formalism and Gibbs measures. 
 
For the rest of the section, we write $X = [0,1] \setminus \Q$ on which the Gauss map is conjugated to the full shift on $\N^\N$. Let $\cM$ be the set of all Borel probability measures on $[0,1]$ and write $\cM_T \subset \cM$ as the collection of $T$ \textit{invariant measures} $\mu$, that is, 
$$\mu(A) = \mu(T^{-1}A) = \sum_{a \in \N} \mu(T_a A)$$
for a Borel set $A \subset X$. Given a function $\phi : X \to \R$ (which we often call a \textit{potential} or an \textit{observable}) and $n \in \N$ let $S_n \phi$ be the \textit{Birkhoff sum} defined at $x \in X$ by
$$S_n \phi(x) = \sum_{k= 0}^{n-1} \phi(T^k(x)).$$
\begin{definition}[Entropy and Lyapunov exponents]
Given $\mu \in \cM_T$, the \textit{Kolmogorov-Sinai entropy} $h_\mu$ and \textit{Lyapunov-exponent} $\lambda_\mu$ of $\mu$ with respect to $T$ are defined by
$$h_\mu = \lim_{n \to \infty} \frac{1}{n}\sum_{\a \in \N^n} - \mu(I_\a) \log \mu(I_\a) \quad \text{and} \quad \lambda_\mu = \int \log |T'| \, d\mu.$$
\end{definition}
We say a function $\phi : X \to \R$ is \textit{locally H\"{o}lder} if there exists constants $C>0$ and $\delta>0$ such that for all $n\in\N$ we have
$$\sup_{\a\in\N^n}\sup\{|\phi(x)-\phi(y)| : x,y\in I_{\a}\}\leq C\delta^n.$$
For a finite word $\a\in\N^n$ we will let $\a^{\infty}$ denote the infinite periodic word of period $n$, which repeats the word $\a$. For a given function $\phi$ we can define its pressure as follows

\begin{definition}[Pressure and equilibrium states]
The \textit{pressure} of a potential $\phi$ is the quantity
$$P(\phi) = \lim_{n\to\infty}\frac{1}{n}\log\left(\sum_{\a \in \N^n}\exp(S_n\phi(\a^{\infty}))\right).$$
Alternatively, we have the variational principle
$$P(\phi)=\sup_{\mu\in\cM_T}\left\{h_\mu+\int\phi\, d\mu:\int\phi \, d\mu>-\infty\right\}$$
and any measure assuming this supremum is called an \textit{equilibrium state} for $\phi$.
\end{definition}

For a given potential $\phi$ there often exists a \textit{unique} invariant measure $\mu = \mu_\phi$, which satisfies a suitable regularity condition with respect to $e^{S_n(\phi(x)) - nP(\phi)}$ known as the \textit{Gibbs} condition: 

\begin{definition}[Gibbs measures] A measure $\mu \in \cM_T$ is a \textit{Gibbs measure} with potential $\phi$, if there is a constant $C \geq 1$ such that for any $n \in \N$ and $\a \in \N^n$ we have
$$C^{-1} e^{S_n \phi(T_\a(x))-nP(\phi)} \leq \mu(I_\a) \leq C e^{S_n \phi(T_\a(x))-nP(\phi)}$$
at any $x \in X$. Notice that by definition $T_\a(x)$ is always an element of the construction interval $I_\a$.
\end{definition}

\begin{remark}\label{rmk:examples}Here are a few well-known examples of Gibbs measures:
\begin{itemize}
\item[(1)] Take some numbers $0 \leq p_a \leq 1$, $a \in \N$, with $\sum_{a \in \N} p_a = 1$, then the measure $\mu$ giving mass $p_{a_1}p_{a_2}\dots p_{a_{n}}$ to the interval $I_{\a}$ for any $\a = a_1\dots a_n$ is called a \textit{Bernoulli measure} for $T$. This measure is a Gibbs measure associated to the potential $\phi$ defined by $\phi(x) = \log p_{a_1(x)}$, where $a_1(x)$ is the first continued fraction digit of $x$. Indeed, the pressure $P(\phi) = 0$, and so the Birkhoff sum $S_n \phi(x)$ reduces the Gibbs comparison to equality:
$$\mu(I_\a) = p_{a_1}p_{a_2}\dots p_{a_{n}} = e^{S_n \phi(T_\a(x))}.$$
\item[(2)] More delicate examples can be found from taking the potential $\phi = -\log |T'|$ for which \textit{Gauss measure} $\mu$ defined by
$$d\mu(x)= \frac{1}{\log 2}\frac{1}{1+x} dx$$
is Gibbs. Gauss measure is the unique $T$ invariant measure which is equivalent to Lebesgue measure.
\item[(3)] If we take some finite $\cA \subset \N$ with  at least $2$ elements and let $s = \dim B(\cA)$, then the $s$-dimensional Hausdorff measure $\cH^s$ restricted to $B(\cA)$ is equivalent to a Gibbs measure with the potential $\phi = -s\log |T'|$ with densities bounded away from $0$ and infinity. This follows for example from \cite[Theorem 5.3]{Falconer1997} when interpreting $B(\cA)$ as a cookie-cutter set.
\end{itemize}
\end{remark}

Gibbs measures are often uniquely determined as invariant measures for the Ruelle transfer operator $\cL_\phi$ determined by the potential. Let $C(X)$ be the space of bounded continuous maps $f : X \to \C$.

\begin{definition}[Ruelle transfer operator] The Ruelle \textit{transfer operator} for a potential $\phi$ is the action $\cL_\phi : C(X) \to C(X)$, defined for $f \in C(X)$ and $x \in X$ by  
$$\cL_\phi f(x) = \sum_{y \in T^{-1}\{x\}} e^{\phi(y)} f(y).$$
The dual operator of $\cL_\phi$ on $\cM_T$ is the action $\cL_\phi^* : \cM \to \cM$, defined for $\nu \in \cM$ and $f \in C(X)$ by
$$\cL_\phi^* \nu (f) := \int \cL_\phi f \, d \nu.$$
\end{definition}

The following result relates the concepts of Gibbs measures, transfer operators and the pressure function. 
\begin{proposition}\label{GE}
Let $\phi:X \to \R$ be a locally H\"{o}lder potential. We have 
$$P(\phi)<\infty \quad \text{if and only if} \quad \sum_{n=1}^{\infty}\exp(\phi((n)^{\infty}))<\infty.$$
Moreover if $P(\phi)=0$ then the following results hold:
\begin{enumerate}
\item
There exists a H\"{o}lder continuous function $h$ bounded away from $0$ and a measure $\mu\in \cM$ such that $\cL_{\phi} h = h$ and $\cL_\phi^*\mu=\mu.$ Moreover the measure $\mu_\phi$ defined by 
$$d\mu_{\phi}(x)=h(x)\,d\mu(x)$$ 
is the unique Gibbs measure for $\phi$. 
\item
If $\int \phi\,d\mu_{\phi}>-\infty$, then $\mu_{\phi}$ is the unique equilibrium state for the potential $\phi$.
\end{enumerate}
\end{proposition}
\begin{proof}
This follows from Proposition 2.1.9, Theorem 2.7.3 and Corollary 2.7.5 in \cite{MauldinUrbanski2003}.
\end{proof}
 For a locally H\"{o}lder potential $\phi$ with $P(\phi)=0$ we will use $\mu_{\phi}$ to denote the unique Gibbs measure for $\phi$. 
One advantage of the transfer operator in our setting is what happens under iteration. Let $n \in \N$ and $f$ be a continuous map we have at any $x \in [0,1]\backslash\Q$ that
$$\cL_\phi^n f(x) = \sum_{y: \, T^n(y) = x} e^{S_n\phi(y)}f(y) = \sum_{\a \in \N^n} e^{S_n(\phi(T_\a(x)))}f(T_\a(x)).$$
Notice that since $T_\a(x) \in I_\a$ and the pressure $P(\phi) = 0$, the definition of Gibbs measure gives
$$w_\a(x) := e^{S_n(\phi(T_\a(x)))}$$
is comparable to $\mu_\phi(I_\a)$ up to the Gibbs constant $C \geq 1$. We now define the family of potentials which we consider.

\begin{definition}\label{L}
We will consider the class of locally H\"{o}lder potentials $\phi$ which satisfy the following properties:
\begin{enumerate}
\item $P(\phi)=0$, $\cL_{\phi}1=1$ and $\phi\leq 0$,
\item There exists $t_c > 0$ such that for all $t\in (-2t_c,2t_c)$ we have that 
$$P(-t\log |T'|+\phi)<\infty \quad \text{and} \quad P(t(\phi-s\log |T'|)+\phi)<\infty,$$ 
where $s$ is the fraction $h_\mu / \lambda_\mu$ of the entropy $h_\mu$ and Lyapunov exponent $\lambda_\mu$ of $\mu = \mu_\phi$.
\end{enumerate}
\end{definition}

\begin{remark}We note that it follows from the Birkhoff ergodic theorem that Gibbs measures are exact dimensional and the number $s = h_\mu/\lambda_\mu$ in Definition \ref{L} is in fact the Hausdorff dimension $\dim \mu$ of the Gibbs measure $\mu = \mu_\phi$, see for example \cite[Theorem 4.4.2]{MauldinUrbanski2003}.
\end{remark}

It should be noted that our first assumption here is not restrictive at all due to the following proposition.
\begin{proposition}
Let $\psi$ be a locally H\"{older} function such that $P(\psi)<\infty$. We can find a locally H\"{o}lder function $\phi$ such that $\phi\leq 0$, $P(\phi)=0$ , $\cL_{\phi}1=1$ and 
$$\mu_{\phi}=\mu_{\psi}.$$
\end{proposition}
\begin{proof}
It follows from Lemma 1 in \cite{sarig2001} that if $P(\psi)<\infty$ then we can find a locally H\"{o}lder $\phi$ such that $\phi\leq 0$, $P(\phi)=0$ , $\cL_{\phi}1=1$ and $\phi-\psi$ is cohomologous to a constant in the class of bounded H\"{o}lder potentials. It follows from Theorem 2.2.7 in \cite{MauldinUrbanski2003} that in this case $\mu_{\phi}=\mu_{\psi}$.
\end{proof}

The following proposition summarises the main results from the thermodynamic formalism on countable Markov shifts we need:
\begin{proposition}\label{td2}
Let $\phi$ be a potential satisfying Definition \ref{L}. Then
\begin{enumerate}
\item
for all bounded continuous functions $f$ we have that
$$\int \cL_{\phi}f\,d\mu_{\phi}=\int f\,d\mu_\phi;$$
\item
let $t_0>0$ and $\psi$ be a locally H\"{o}lder continuous potential such that $P(t\psi+\phi)<\infty$ for all $t\in (-t_0,t_0)$. If we define $Q:(-t_0,t_0)\to\R$ by 
$$Q(t)=P(t\psi+\phi),$$ 
then $Q$ is analytic and convex on $(-t_0,t_0)$ and $Q'(t)=\int \psi\,d\mu_{t\psi+\phi}$.
\end{enumerate}
\end{proposition}
\begin{proof}
The first part of the proposition is immediate from the condition that $\cL_{\phi}1=1$ and the fact that $\mu_{\phi}$ is then a fixed point for the dual operator. The second part of the proposition follows from Corollary 4 in Sarig \cite{sarig2003}.
\end{proof}
\begin{remark}
We also include the case where we restrict $T$ to $\cB(\cA)$ for a finite $\cA \subset \N$. Then we let $\phi:B(\cA)\to\R$ be a H\"{o}lder continuous potential satisfying the conditions of Definition \ref{L}. All the above results hold in this setting with the added advantage that the pressure is never infinite.
\end{remark}

\section{Large deviation bounds for Gibbs measures}
\label{sec:largedeviations}

Let $X_1,X_2,\ldots$ be identically distributed independent random variables with expectation $0$. In probability theory large deviations in their simplest form examine the rate at which 
$$\P({|X_1+\cdots+X_n|>\eps n})$$ 
decays as $n \to \infty$ for some $\eps>0$. \textit{Cram\'er's theorem} states that if the moment generating function is always finite then this decay is exponential, see for example Section 1.9 in \cite{Durrett1995}. Large deviations also appear in dynamics and it turns out that the theory can be developed for certain invariant measures for hyperbolic dynamical systems, see for example the paper by Young \cite{Young1990} in the compact setting. In this section we prove results of this form for the Gauss map we need for the Rajchman property.

 In the finite state space case the results follow from the much more general results in \cite{Kifer1990,Lopes1990,Young1990} so we concentrate on proving the results in the countable state space. This is by no means the first time large deviations have been considered for the countable shift, for example see Section 6 of the paper Kifer, Peres and Weiss \cite{KiferPeresWeiss2001} or Yuri \cite{Yuri2005}. However, they consider observables which are bounded whereas we are interested in the observable $-\log |T'|$ which is unbounded.

Let $\phi$ be a locally H\"{o}lder potential such that $\phi \leq 0$ and $P(\phi)=0$. We denote by $\mu_{\phi}$ the Gibbs measure for $\phi$. For simplicity of the notation, given a locally H\"older observable $f$, write
$$\alpha(f) := \int f \, d\mu_\phi.$$
Our main large deviation result is the following
\begin{theorem}\label{ld1}
Let $f$ be any locally H\"older observable. Suppose 
$$\{t \in \R:P(t f+\phi)<\infty\}$$ 
contains a neighbourhood of the origin. Then for any $\epsilon>0$ there exists $\delta = \delta(\epsilon)> 0$ and $n_1=n_1(\epsilon)\in\N$ such that for $n\geq n_1$
$$\mu_{\phi}(\{x \in [0,1] :|S_n f(x)-n\alpha(f)|>n\epsilon\}) \leq e^{-n\delta}.$$
\end{theorem}
As a corollary to this we have the large deviation result for the observables we use. Write $\psi=-\log |T'|$ for the observable corresponding to the Lyapunov exponent. That is, we have
$$\lambda = \lambda_{\mu_\phi} = -\alpha(\psi).$$
We let $s = h_{\mu_\phi} / \lambda_{\mu_\phi}$, which is the Hausdorff dimension of $\mu_\phi$.

\begin{corollary}\label{cor:largedeviations}
\begin{enumerate}
\item\label{sum}
Suppose that $\{t \in \R:P(t\psi+\phi)<\infty\}$ contains a neighbourhood of the origin. Then for any $\epsilon>0$ there exists $\delta = \delta(\eps) > 0$ and $n_1 = n_1(\eps) \in \N$ such that for all $n \geq n_1$ we have 
$$\mu_{\phi}(\{x \in [0,1]: |S_n\psi(x)+n\lambda|\geq n\epsilon\})\leq e^{-n\delta}.$$
\item\label{quotient}
Suppose that $\{t \in \R :P(t(\phi-s \psi)+\phi)<\infty\}$ contains a neighbourhood of the origin.  Then for any $\epsilon>0$ and for all $n \geq n_1$ we have 
$$\mu_{\phi}\left(\left\{x \in [0,1]: \left|\frac{S_n\phi(x)}{S_n \psi(x)}-s\right|>\epsilon\right\}\right)\leq e^{-n\delta}.$$
\end{enumerate}
\end{corollary}
\begin{proof}
Part \ref{sum} is Theorem \ref{ld1} with $f = \psi$. To prove part \ref{quotient} we let $f=\phi-s \psi$ and note that here $\alpha(f)=0$. We then have that for $\epsilon>0$ if $x$ satisfies that 
$$\left|\frac{S_n\phi(x)}{S_n \psi(x)}-s\right|>\epsilon$$
giving $|S_n f(x)| > \eps |S_n \psi(x)|$. By Lemma \ref{lma:growth} and Lemma \ref{lma:goldengrowth} we have that
$$|(T^k)'(x)|\geq c_0^2\theta^{2k}$$
where, recall, $\theta > 1$ was the Golden ratio and $c_0 > 0$ a universal constant. Then we obtain by the definition of the Birkhoff sum $S_n \psi(x)$ that
$$|S_nf(x)|>\epsilon |S_n \psi(x)| = \eps \cdot \sum_{k = 0}^{n-1} \log |(T^k)'(x)| \geq \epsilon n\log (c_0^2 \theta^{n+1}).$$
Choosing $n_1 = n_1(\eps)$ larger in Theorem \ref{ld1} such that $\log (c_0^2 \theta^{n+1}) \geq 1$ holds for all $n \geq n_1$, the result is now an application of Theorem \ref{ld1} since $\alpha(f) = 0$. \end{proof}
We now give the proof of Theorem \ref{ld1}. We start with the following simple lemma which exploits the convexity of the pressure function.
\begin{lemma}
\label{lma:dev1}
If there exists $t_0 > 0$ such that $(-t_0,t_0)\subset\{t \in \R:P(t f+\phi)<\infty\}$ and $\alpha\neq\alpha(f)$, then there exists $t \in (-t_0,t_0)$ such that
$P(t(f-\alpha)+\phi)<0$.
\end{lemma}
\begin{proof}
For simplicity we assume that $\alpha<\alpha(f)$ the case where $\alpha > \alpha(f)$ can be handed analogously. We define the function $Q:(-t_0,t_0)\to\R$ by 
$$Q(t)=P(tf+\phi)$$ 
and note that by Proposition \ref{td2} this function will be analytic,  convex, $Q'(t)=\int f\,d \mu_t$ where $\mu_t$ is the unique Gibbs state for $tf+\phi$ and $Q(0)=0$ with $Q'(0)=\alpha(f)$. 
 Thus there exists $\alpha_1 \in [\alpha,\alpha(f)]$ and $-t_0<t<0$ such that $Q'(t)=\alpha_1$. Now consider the case $\alpha_1=\alpha(f)$. Because $\alpha < \alpha(f) = \alpha_1$ we have that
$$P(t(f-\alpha)+\phi)=Q(t)-\alpha t< Q(t)-\alpha_1t=Q(0)=0$$
by the convexity of $Q$. Moreover, in the case $\alpha_1<\alpha(f)$ we have by the convexity of $Q$ that
$$P(t(f-\alpha)+\phi)=Q(t)-\alpha t\leq Q(t)-\alpha_1t<Q(0)=0.$$
\end{proof}

\begin{proof}[Proof of Theorem \ref{ld1}] By Lemma \ref{lma:dev1}, we can fix $\alpha \neq \alpha(f)$ and $t \in \R$ such that 
$$\delta_1:=-P(t(f-\alpha)+\phi) > 0.$$ 
Let us first assume $\alpha < \alpha(f)$; the other case is symmetric. Note that in this case we must have $t < 0$. Indeed, the map $W$ defined by $W(v) = P(v(f-\alpha)+\phi)$ is convex by $\alpha < \alpha(f)$, the value $W(0) = P(\phi) = 0$, the derivative $W'(0) = \alpha(f) - \alpha > 0$ and that $W(t) = -\delta_1 < 0$. As $f$ is locally H\"older, we may fix $n_1 \in \N$ such that for any $n\geq n_1$, $\a\in\N^n$ and any $x,y\in I_{\a}$ we have
that $|S_nf(x)-S_nf(y)|\leq \frac{n\delta_1}{4t}$ and also that
$$\left|\frac{1}{n}\log\left(\sum_{\a\in\N^n}\exp(S_n (t(f-\alpha)+\phi)(\a^\infty))\right)-P(t(f-\alpha)+\phi)\right|\leq \frac{\delta_1}{4},$$ 
by the definition of the pressure, where recall $\a^\infty \in \N^\N$ is the periodic infinite word repeating $\a$ thought as a point in $X$. We also have that by the Gibbs property of $\mu_{\phi}$ we have for all $n \in \N$ that
$$\mu_{\phi}(I_{\a})\leq C\exp(S_n \phi(\a^{\infty}))$$
where $C$ is the Gibbs constant of $\mu_\phi$. We let
$$\cC_n=\{\a\in\N^n : S_nf(x)\leq\alpha n\text{ for some }x\in I_{\a}\}.$$
We can then calculate for $n\geq n_1$ that
\begin{align*}
\mu_{\phi}(\{x \in [0,1]: S_nf(x) \leq n\alpha\})&\leq\sum_{\a \in \cC_n}\mu(I_{\a})\\
&\leq C\sum_{\a \in \cC_n}\exp(S_n \phi(\a^\infty))\\
&\leq Ce^{n\delta_1/4}\sum_{\a \in \cC_n}\exp(S_n (t(f-\alpha)+\phi)(\a^{\infty}))\\
&\leq Ce^{n\delta_1/4}\sum_{\a\in\N^n}\exp(S_n (t(f-\alpha)+\phi)(\a^{\infty}))\\
&\leq Ce^{n\delta_1/2}\exp(nP(t(f-\alpha)+\phi))\\
&= Ce^{-n\delta_1/2}.
\end{align*}
In the case $\alpha > \alpha(f)$ we obtain a symmetric computation with $t > 0$ we use to define $\delta_1$. Thus for any $0 < \eps < |\alpha - \alpha(f)|$, we obtain
$$\mu_{\phi}(\{x \in [0,1]: |S_nf(x)-n\alpha(f)| > \eps n\}) \leq 2Ce^{-n\delta_1/2}$$
which completes Theorem \ref{ld1} as now any $\delta < \delta_1/2$ is enough for the claim.
\end{proof}

Moreover, the second assumption of Definition \ref{L} can be characterised by the fatness of the tail. This will be shown in the following result, which explains why the assumption \eqref{pt} is necessary and sufficient for $\phi$ to satisfy the conditions of parts (1) and (2) of Corollary \ref{cor:largedeviations} 

\begin{proposition}\label{1term}
Let $\phi : X \to \R$ be a locally H\"{o}lder potential with pressure $P(\phi)=0$, $\cL_{\phi}1=1$ and $\phi\leq 0$ with $s \in [0,1]$ the Hausdorff dimension of $\phi$. Then the following are equivalent
\begin{itemize}
\item There exists $\delta>0$ such that
$$\mu_{\phi}(\{x \in X :a_1(x)\geq n\})=O(n^{-\delta})$$

\item There exists $t_0>0$ such that for all $-t_0<t<t_0$ we have that
$$P(t\psi+\phi)<\infty\text{ and }P(t(\phi-s\psi)+\phi)<\infty.$$
\end{itemize}
\end{proposition}
\begin{proof}
Firstly suppose there exists $\delta>0$ such that
$$\mu_{\phi}(\{x \in X :a_1(x)\geq n\})=O(n^{-\delta}).$$
Thus by the Gibbs property of $\mu_{\phi}$ we have that
$$\sum_{k=n}^{\infty}e^{\phi((k)^{\infty})}=O(n^{-\delta}),$$
which implies that for any $0 < t < \delta/2$ we have
\begin{equation}\label{f1}
\sum_{k=1}^{\infty}e^{\phi((k)^{\infty})}k^{2t}<\infty
\end{equation}
and
\begin{equation}\label{f2}
\sum_{k=1}^{\infty}e^{(t+1)(\phi((k)^{\infty}))-st\psi((k)^{\infty})}<\infty.
\end{equation}
Since by Lemma \ref{lma:growth} we have $k^2 \leq |T'((k)^\infty)| \leq 4k^2$ so $e^{-st\psi((k)^{\infty})} \leq 4^{st} k^{2st} \leq 4^{st} k^{2t}$. By Proposition \ref{GE} it follows from \eqref{f1} and \eqref{f2} that for $t$ sufficiently small 
$$P(t\psi+\phi)<\infty\text{ and }P(t(\phi-s\psi)+\phi)<\infty$$
as claimed. For the other direction if for $t > 0$ sufficiently small $P(t\psi+\phi)<\infty$ we have that
$$\sum_{k=1}^{\infty}e^{\phi((k)^{\infty})}k^{2t}<\infty$$
from which it follows that for that 
$$\sum_{k = n}^\infty e^{\phi((k)^{\infty})} = O(n^{-2t})$$
as $n \to \infty$. Indeed if $e_k := e^{\phi((k)^{\infty})}$ and $c := \sum_{k=1}^\infty e_k k^{2t} < \infty$, then for all $n \in \N$ we have $\sum_{k=n}^\infty e_k k^{2t}\leq c$ and so as $t > 0$ we have
$$\sum_{k=n}^\infty e_k\leq \sum_{k=n}^\infty e_k k^{2t}n^{-2t}\leq cn^{-2t}.$$ 
Thus by the Gibbs property of $\mu_{\phi}$ we have
$$\mu_{\phi}(\{x \in X :a_1(x)\geq n\}) \leq \sum_{k = n}^\infty Ce^{\phi((k)^{\infty})} =\mathit{O}(n^{-2t})$$
and the result follows. One can observe that in this direction of the equivalence we do not need the second condition on the pressure.
\end{proof}

\begin{remark}\label{rmk:large}It should be noted that if $\mu_\phi$ did not satisfy condition \eqref{pt} then the large deviation result could not hold. In this case we would have that
$$\limsup_{n\to\infty} n^\delta \mu_{\phi}(\{x \in X:a_1(x)\geq n\})=\infty$$
then as $e^{\phi(x)}$ is at least $4a_1(x)^2$ by Lemma \ref{lma:growth} we have that Corollary \ref{cor:largedeviations} could not hold as $\mu_{\phi}(\{x \in X: |S_n\psi(x)+n\lambda|\geq n\epsilon\})$ would decay subexponentially.
\end{remark}

\begin{remark}
We also remark that in the proofs of the results in this section, such as Theorem \ref{ld1}, we do not use any specific properties of the Gauss map itself, only the symbolic coding with the countable Markov shift, so the results could be extended to other Markov maps with similar coding.
\end{remark}

\section{Decomposition of the Gibbs measure} 
\label{sec:decomposition}

In this section we provide an important decomposition of the Gibbs measure on each generation $n$. This will be done by creating cylinders from the `good sets' for the large deviations, that is, where we can control the exponential decay of $|T_\a'(x)|$ and the measure $\mu(I_{\a})$ with respect to the length $|I_\a|$. Thanks to the large deviation bounds, the exceptional sets to these regular parts is exponentially small. Fix $\eps > 0$ from now on and suppress it from the notation. We will only make assumptions on $\eps$ when dealing with the proof of Theorem \ref{thm:main} in Section \ref{sec:proof}.

We begin by choosing $n_0$ sufficiently large such that the large deviation bounds in Corollary \ref{cor:largedeviations} and that some technical conditions are true. It should be noted that we keep the Gibbs measure $\mu$ and its potential $\phi$ fixed.

\begin{choice} Choose an \textit{even} number $n_0 \in \N$ such that the following are true
\begin{itemize}
\item[(1)] If $n_1 = n_1(\eps/2)$ is the threshold from large deviations result Corollary \ref{cor:largedeviations} with $\eps/2$ in the place of $\eps$, then
\begin{align}\label{n01}n_0/2 > n_1.\end{align}
\item[(2)] If $C$ is the Gibbs constant for $\mu$, and $\theta$ is the Golden ratio, and $c_0$ is the constant from Lemma \ref{lma:goldengrowth}, we assume
\begin{align}\label{n02}\frac{\log 4}{n_0/2} < \eps /2 \quad \text{and} \quad \frac{\log 4C^2}{\log (c_0^2\theta^{n_0})} < \eps/2.\end{align}
\item[(3)] If $\delta = \delta(\eps/2) > 0$ is the rate function for the large deviations Corollary \ref{cor:largedeviations}, then
\begin{align}\label{n03}\frac{e^{-\delta n_0/2}}{1-e^{-\delta}} < e^{-\delta n_0/4}.\end{align}
\item[(4)] If $\lambda$ is the Lyapunov exponent of $\mu$, then
\begin{align}\label{n04}e^{n_0/2} \geq 2\pi e^{\lambda + 2\lambda\eps}, \quad e^{\eps \lambda n_0} \geq 2C\pi e^{(1+2s)\lambda} \quad \text{and} \quad \frac{1}{192} e^{(\frac{\lambda}{2}-2\eps) n_0} \geq 1.\end{align}
\end{itemize}
\end{choice}

The reason for (1) in this choice is that we will want to obtain large deviation bounds at the halved generation $n/2$ for some $n \geq n_0$.

\begin{construction}[Regular intervals] We will now construct a collection $\cR_n$ of words $\a \in \N^n$ where we have good control on the growth of the continuants $q_n(\a)$ and the weights $\mu(I_\a)$. With the definition of $n_0$ in mind, for $n \geq n_0$ write the \textit{$n$-regular set} for large deviations
$$A_n(\eps) := \Big\{x \in [0,1] : \Big|\frac{1}{n}S_n \psi(x) +\lambda\Big| < \eps, \Big|\frac{S_n \phi(x)}{S_n \psi(x)} - s\Big| < \eps\Big\}$$
and define the collection of $n$-\textit{regular words} by
$$\cR_n := \bigcap_{k = \lfloor n /2 \rfloor}^{n} \{\a \in \N^n : I_{\a|_k} \subset A_k(\eps)\}$$
and let $R_n \subset [0,1]$ be the corresponding union of intervals in $\cR_n$.
\end{construction}

The reason to take intersections in the construction is to have have the $n$-regular intervals from $\cR_n$ \textit{nested} in the sense that if $\a \in \cR_n$, then we know the `$k$-regularity' for the interval $I_{\a|_k}$ for smaller $k$ up to $n/2$. This space is useful when dealing with large intervals for example in the proof of Theorem \ref{thm:main}.

The advantage of large deviations for Lyapunov exponents and Hausdorff dimension is that we obtain bounds for the continuants $q_n(\a)$ and the weights $w_\a(x) = e^{S_n\phi(T_\a(x))}$ in the terms of $e^{\lambda n}$ when $\a \in \cR_n$. The following lemma summarises all the comparisons we have:

\begin{lemma}
\label{lma:regular}
Fix $n \geq n_0$, $\a \in \cR_n$ and $\lfloor n /2 \rfloor \leq k \leq n$. Then the following comparisons hold
\begin{itemize}
\item[(1)] the continuant $q_k(\a)$ satisfies
$$e^{(\lambda-\eps)k} \leq q_k(\a)^2 \leq 4 e^{(\lambda+\eps)k};$$
\item[(2)] the length $|I_{\a|_k}|$ and the derivative $T_{\a|_k}'$ satisfies
$$\tfrac{1}{16}e^{(-\lambda-\eps)k} \leq |I_{\a|_k}| \leq e^{(-\lambda+\eps)k} \quad \text{and} \quad \tfrac{1}{16}e^{(-\lambda-\eps)k} \leq |T_{\a|_k}'| \leq e^{(-\lambda+\eps)k};$$
\item[(3)] if $x \in [0,1]$, then the weight $w_{\a|_k}(x) = e^{S_k\phi(T_{\a|_k}(x))}$ satisfies
$$e^{(-s\lambda-3\lambda\eps) k} \leq w_{\a|_k}(x) \leq e^{(-s\lambda+3\lambda\eps)k};$$
\item[(4)] if $C \geq 1$ is the Gibbs constant of $\mu$, then the measure $\mu(I_{\a|_k})$ satifies
$$C^{-1}e^{(-s\lambda-3\lambda\eps) k} \leq \mu(I_{\a|_k}) \leq Ce^{(-s\lambda+3\lambda\eps)k}.$$
\end{itemize}
\end{lemma}

\begin{proof}
We do this just in the case $k = n$ as the proof is the same for the other cases. Fix $x \in [0,1]$ and $\a \in \N^n$. Then by the chain rule
\begin{align}
\label{ComparingTandS}
\log |T_\a'(x)| = \sum_{j = 0}^{n-1} \log \frac{1}{|T' (T^j(T_\a(x))|}= S_n \psi (T_\a(x)).
\end{align}
Lemma \ref{lma:growth} gives $\tfrac{1}{4}q_n(\a)^{-2} \leq |T_\a'(x)| \leq q_n(\a)^{-2}$ so
\begin{align}\label{comparison}
e^{S_n \psi (T_\a(x))}  \leq q_n(\a)^{-2} \leq 4e^{S_n \psi (T_\a(x))}.
\end{align}
If $T_\a(x) \in A_n(\eps)$, then
$$e^{(-\lambda-\eps)n} \leq e^{S_n \psi (T_\a(x))} \leq e^{(-\lambda+\eps)n}.$$
Thus by \eqref{comparison}, we obtain
$$e^{(\lambda-\eps)n} \leq q_n(\a)^2 \leq 4 e^{(\lambda+\eps)n},$$
which proves (1). Now (2) follows when recalling the distortion bounds for the lengths $|I_{\a|_k}|$ and $T_{\a|_k}'$ from Lemma \ref{lma:growth}. Moreover, as $S_n \psi(T_\a(x)) \leq 0$ we have
$$e^{(s+\eps)S_n \psi (T_\a(x))} \leq e^{S_n \phi (T_\a(x))} \leq e^{(s-\eps)S_n \psi (T_\a(x))}.$$
Thus by using crude bounds $\eps,s \leq 1$ and $\lambda > 1$ we obtain
$$e^{(-s\lambda-3\lambda\eps) n} \leq w_\a(x) \leq e^{(-s\lambda+3\lambda\eps)n}.$$
which is (3). This also gives (4) by the Gibbs property that links $\mu(I_\a)$ to $w_\a(x)$.
\end{proof}

As we are intersecting with sets $A_k(\eps)$ whose $\mu$ measure is exponentially close to one, we do not lose much in the measure:

\begin{lemma}\label{lmaregular2}
For all $n \geq n_0$ we have
$$\mu([0,1] \setminus R_n) \leq e^{-\delta n/4}$$
where $\delta = \delta(\eps/2) > 0$ is the rate function from the large deviations with $\eps/2$.
\end{lemma}

\begin{proof}
We wish to show that 
\begin{align}\label{containsintersection}\bigcap_{k = \lfloor n /2 \rfloor }^n A_k(\eps/2) \subset R_n.\end{align}
This is enough for us as it yields by the geometric series formula that
$$\mu([0,1] \setminus R_n) \leq \sum_{k = \lfloor n /2 \rfloor}^n \mu([0,1] \setminus A_n(\eps/2)) \leq \frac{e^{-\delta n/2}}{1-e^{-\delta}} < e^{-\delta n/4},$$
where $\delta := \delta(\eps/2)/2$ since $\lfloor n /2 \rfloor \geq n_1(\eps/2)$ by \eqref{n01} of the choice of $n_0$ since $n_0$ is even.

To prove \eqref{containsintersection} let $\a \in \N^n$ be a word such that $T_{\a}(x) \in A_k(\eps/2)$ for all $k \in \lfloor n /2 \rfloor,\dots,n$. We wish to prove that $I_{\a|_k} \subset A_k(\eps)$ for all these $k$. Fix $k$ and a point $y \in [0,1]$ and let us prove that $T_{\a|_k}(y) \in A_k(\eps)$. Since $T_\a(x) \in A_k(\eps/2)$ we have by Lemma \ref{lma:growth} that
$$\Big|\frac{1}{k}S_k \psi(T_{\a|_k}(y)) +\lambda\Big| < \eps/2 + \frac{\log 4}{k} < \eps$$
and for the Gibbs constant $C \geq 1$ in addition the Lemma \ref{lma:goldengrowth} shows that
$$\Big|\frac{S_k \phi(T_{\a|_k}(y))}{S_k \psi(T_{\a|_k}(y))} - s\Big| < \eps/2 + \frac{\log 4C^2}{|S_k \psi(T_{\a|_k}(y))|} \leq \eps/2 + \frac{\log 4C^2}{\log (c_0^2\theta^{2k})} < \eps$$
by the choice of $n_0$ (the property \eqref{n02}). More precisely, we used the triangle inequality after applying \eqref{ComparingTandS} to obtain $S_k \psi(T_\a(x)) = \log |T_{\a|_k}'(z)|$, where $z = T_{\sigma^k \a}(x)$ and $\sigma$ is the shift, which yields
\begin{align*}
S_k \psi (T_{\a|_k}(y)) - S_k \psi (T_{\a}(x)) = \log \frac{|T'_{\a|_k}(y)|}{|T'_{\a|_k}(z)|}
\end{align*}
and finishing the proof with Lemma \ref{lma:growth} which allows us to compare the derivatives at points $y$ and $z$. Thus $I_{\a|_k} \subset A_k(\eps)$ and we are done as $k = \lfloor n /2 \rfloor,\dots,n$ is arbitrary.
\end{proof}

\section{Proof of the main result}
\label{sec:proof}

\subsection{Overview} 
\label{subsec:overview}

We will now give some heuristics and general strategy we follow during the proof.
\begin{itemize}
\item[(1)] We first choose a natural number $n \geq n_0$ so large that $|\xi|^{-1}$ is approximately the exponential number $e^{-c\lambda n}$ for some constant $c > 0$, where $\lambda$ is the Lyapunov exponent of the Gibbs measure  $\mu = \mu_\phi$. Then the proof reduces to verify that the Fourier transform $|\widehat{\mu}(\xi)|$ has an exponential bound $e^{-\eta c\lambda n}$ for some constant $\eta > 0$.

\item[(2)] The first step is to use the invariance of the Gibbs measure $\mu = \mu_\phi$ under the dual operator $\cL_\phi^*$ and apply the Ruelle transfer operator $n$ times to the oscillation $x \mapsto e^{-2\pi i \xi x}$. This, by the definition of the transfer operator, yields that $\widehat{\mu}(\xi)$ becomes a summation over $n$th pre-images of $T$, that is, a summation over the words $\a \in \N^n$ of certain oscillative $\mu$ integrals.

\item[(3)] This is the point where we invoke the decomposition of the Gibbs measure to the regular and irregular part using the $n$-regular cylinders $\cR_n$ given by the large deviations. Since not much $\mu$ mass enters outside of $R_n$, we obtain an exponential decay $e^{-\delta n/2}$ for that part from the large deviations and Lemma \ref{lmaregular2}; see Section \ref{subsec:split} below. Thus we are left to just study the summation over the regular words $\a \in \cR_n$ of oscillative $\mu$ integrals.

\item[(4)]  We then exploit the H\"older bounds for the measure $\mu(I_\a)$ on regular intervals $I_\a$ and the decay rate for $T_\a'(x)$ to transfer the estimates from the oscillative $\mu$ integrals for the Gibbs measure to oscillative $L^2$ Lebesgue integrals. The price we pay is that the there is an exponentially increasing term fighting against the Lebesgue integrals, so we need a rapid enough decay to counter this multiplicative error.

\item[(5)] When we deal with the Lebesgue measure and $n$-regular words $\a$ for which we can control the growth of the continuants $q_n(\a)$, we can then rely on \textit{stationary phase} type inequalities that Kaufman also used, which, when invoking the rapid expansive nature of the Gauss map, yield an exponential decay bound $e^{-\lambda n / 2}$ in our case. This decays rapidly enough for us to counter the error produced in the step (4).

\end{itemize}

\subsection{Fixing parameters} 
\label{subsec:fix}

We begin the proof by fixing the the parameters $\eps > 0$ for large deviations and the number $n \geq n_0$ related to $|\xi|^{-1}$, recall choice of $n_0$ in the beginning of Section \ref{sec:decomposition}. We will invoke the convention $f(x) \ll g(x)$ if there is a constant independent of $x$ such that $f(x) \leq Cg(x)$. 

Let $\mu = \mu_\phi$ be Gibbs measure with entropy $h = h_\mu$, Lyapunov exponent $\lambda = \lambda_\mu$ and Hausdorff-dimension $s =h/\lambda = \dim \mu > 1/2$. Then the number
$$\eta_s := \frac{2s^2-s}{(4-s)(1+2s)} > 0.$$ 
Let us fix $\eps > 0$ such that 
\begin{align}\label{epschoice}\eps < 1/26 \quad \text{and}\quad \eta_s-38\eps > 0.\end{align} 
This now fixes $n_0 = n_0(\eps) \in \N$ which we defined in the Section \ref{sec:decomposition} before decomposing the Gibbs measure. 

To prove the main Theorem \ref{thm:main} it is enough to show that
\begin{align}\label{fourierbound}|\widehat{\mu}(\xi)| \ll |\xi|^{-\eta_s+38 \eps} + |\xi|^{-\frac{\delta}{12\lambda}},\end{align}
where $\delta = \delta(\eps/2) > 0$ is the exponential rate from the large deviations with $\eps/2$. Depending on how small $\eps$ is, one of the terms in \ref{fourierbound} will dominate the other but in both cases we obtain polynomial decay for the Fourier transform.

Let us write for the rest of the paper $u := |\xi|$. Assume $u$ is initially so large that we can choose $n \geq n_0$ such that
\begin{align}\label{choiceofn}
e^{(1+2s)\lambda n} \leq u < e^{(1+2s)\lambda(n+1)}.
\end{align}

\subsection{Splitting the transfer operator} 
\label{subsec:split}

Define $h : [0,1] \to \C$ by $h(y) = \exp(-2\pi i \xi y)$. Since $\mu$ is a Gibbs measure with potential $\phi$, we have that $\mu$ satisfies invariance condition $\cL_\phi^* \mu = \mu$ and if applied $n$ times to $h$ gives
$$\hat{\mu}(\xi) = \int \cL_\phi^n h(x) \, d \mu(x).$$
By the definition of the transfer operator
$$\cL_\phi^n h(x) = \sum_{\a \in \N^n} w_\a(x) \exp(-2\pi i \xi T_\a(x))$$
where, recall, the weight $w_\a(x) := e^{S_n \phi(T_\a(x))}$. Now decompose the summation 
$$\cL_\phi^n h(x) = f(x) + g(x)$$ 
to the \textit{regular part} $f(x)$ and the \textit{irregular part} $g(x)$ according to $\cR_n$ as follows
\begin{align}\label{split}f(x) =  \sum_{\a \in \cR_n} w_\a(x) \exp(-2\pi i\xi T_\a(x)) \quad \text{and} \quad g(x) =  \sum_{\a \in \N^n \setminus \cR_n} w_\a(x) \exp(-2\pi i\xi T_\a(x)).
\end{align}
Hence in order to estimate $|\widehat{\mu}(\xi)|$ we need to estimate $\int |f| \, d\mu$ and $\int |g| \, d\mu$. Both terms will contribute to the decay rate of the Fourier transform, depending on the value $\eps > 0$. For the irregular part, we can immediately use the measure bound from Lemma \ref{lmaregular2}. Indeed since $e^{(1+2s)\lambda n} \leq u$ and by the Gibbs property of $\mu$ we have $w_\a(x) \leq C\mu(I_\a)$ so inputting $e^{(1+2s)\lambda n} \leq u$ from \eqref{choiceofn} we have
\begin{align}\label{g}\int|g(x)| \, d \mu(x) \leq C\mu([0,1] \setminus R_n) \leq Ce^{-\delta  n/4} \ll u^{-\frac{1}{4(1+2s)\lambda}\delta}\ll u^{-\frac{\delta}{12 \lambda}}\end{align}
since $\lambda > 1$ and $s \leq 1$. For the regular part, on the other hand, we will obtain the bound 
\begin{align}\label{mainbound}\int |f(x)| \, d \mu(x) \ll u^{-\eta_s + 38 \eps} + u^{-\frac{\delta}{12 \lambda}}\end{align}
with the same term on the right. Since $u = |\xi|$, this yields the desired decay and Theorem \ref{thm:main}. The proof of this bound is now split into the Sections \ref{subsec:gibbslebesgue}, \ref{subsec:L2} and \ref{subsec:completion} below.

\subsection{From Gibbs to Lebesgue}
\label{subsec:gibbslebesgue}

To estimate the oscillative integral $\int |f(x)| \, d\mu(x)$ we rely on a method to transfer the problem on integrating over the Gibbs measure $\mu$ back to Lebesgue measure. In the case of discrete sums of point masses, this kind of approach reduces to \textit{large sieve inequalities}, see for example inequality (4) in Baker's paper \cite{Baker1981}. The analytic version used for example by Baker requires classically H\"older bounds from the original measure $\mu$, see for example the inequalities used by Kaufman \cite[Section 5]{Kaufman1980}. In our case, Gibbs measures do not have H\"older bounds but on the $n$-regular cylinders we do have control for them. Thus we will give the following version using the $L^2$ norm $\|f\|_2$ with respect to Lebesgue measure. Note that here the proof applies to any $C^1$ function $f$ instead of the specific $f$ we consider:

\begin{lemma}
\label{lma:largesieve}
Suppose we are given numbers $\alpha,\beta > 0$ such that $\gamma = \alpha+\beta \geq 1$ and let $f : [0,1] \to \C$ be $C^1$ with $\|f\|_{\infty} \leq 1$ and $\|f'\|_\infty \leq e^{\alpha \lambda n}$. Then
$$\int |f(x)| \, d \mu(x) \ll e^{(-\beta \lambda+\eps\gamma ) n} +  \|f\|_2^2 e^{((1-s)\alpha \lambda + (3-s)\beta \lambda  +3\lambda\eps \gamma) n} + e^{-\delta\gamma n/4}.$$
\end{lemma}

\begin{proof}
Write $m = \lfloor \gamma n \rfloor$ and 
$$\cR_{m}' = \{\a \in \cR_{m} : \sup\{ |f(x)| : x \in I_\a\} \geq 2 e^{\alpha \lambda n} e^{(-\lambda+\eps)m}\}.$$
By the condition $\alpha + \beta \geq 1$ and $n \geq n_0$, we have $m \geq n_0$ so we may apply all the large deviation theory for the words in $\cR_m$. Let $\a \in \cR_m'$. Since $f$ is $C^1$ and by Lemma \ref{lma:regular}(2) we have $|I_\a| \leq e^{(-\lambda+\eps)m}$ then by the mean value theorem
$$|f(x)| \geq 2 e^{\alpha \lambda n} e^{(-\lambda+\eps)m} - \|f'\|_\infty |I_\a| \geq e^{\alpha \lambda n} e^{(-\lambda+\eps)m}$$
for all $x \in I_\a$. Thus by Lemma \ref{lma:regular}(2) we have
\begin{align*}\int |f(x)|^2 \, dx &\geq \sum_{\a \in \cR_m'} e^{2\alpha \lambda n+2(-\lambda+\eps)m} |I_\a|\\
& \geq  |\cR_m'| \cdot e^{2\alpha \lambda n+(-2\lambda+2\eps)m} \cdot \tfrac{1}{16}e^{(-\lambda-\eps)m}\\
& \geq  \tfrac{1}{16} |\cR_m'| \cdot e^{2\alpha \lambda n+(-3\lambda+\eps)m}.
\end{align*}
This yields by Lemma \ref{lma:regular}(4) and $\|f\|_\infty \leq 1$ that for the Gibbs constant $C > 0$ the integral
\begin{align*} \int_{R_m'} |f(x)| \, d \mu(x) \leq  \sum_{\a \in \cR_m'} \mu(I_\a) & \leq 16C \|f\|_2^2 e^{-2\alpha \lambda n+(3\lambda-\eps)m} e^{(-s\lambda+3\lambda\eps)m} \\
&\leq 16C\|f\|_2^2 e^{-2\alpha \lambda n+(3-s)\lambda m +3\lambda\eps m}\\
& \ll \|f\|_2^2 e^{(-2\alpha \lambda +(3-s)\lambda \gamma  +3\lambda\eps \gamma) n}\\ 
& = \|f\|_2^2 e^{((1-s)\alpha \lambda + (3-s)\beta \lambda  +3\lambda\eps \gamma) n}.\end{align*}
Moreover, by the definition of $\cR_m'$ we have as $\alpha + \beta = \gamma$ that
$$\int_{R_m \setminus R_m'} |f(x)| \, d\mu(x) \leq 2 e^{\alpha \lambda n+(-\lambda+\eps)m} \ll e^{(-\beta \lambda+\eps\gamma ) n}.$$ 
Finally, the remaining part we use $\|f\|_\infty \leq 1$ to obtain a measure bound
$$\int_{[0,1] \setminus R_m} |f(x)| \, d\mu(x) \leq \mu([0,1] \setminus R_m) \leq e^{-\delta m/4} \ll e^{-\delta\gamma n/4}$$
by Lemma \ref{lmaregular2}.
\end{proof}

We apply Lemma \ref{lma:largesieve} for the regular part $f(x) = \sum_{\a \in \cR_n} w_\a(x) \exp(-2\pi i\xi T_\a(x))$, recall \eqref{split}. The first term of the bound in Lemma \ref{lma:largesieve} already decays exponentially as $n \to \infty$. However, the larger $\beta > 0$ we choose, the faster the exponential term in the front of $\|f\|_2^2$ grows. Thus we need to find a suitable choice of $\beta$ so the growth will not terminate the decay of $\|f\|_2^2$ as $n \to \infty$. The number $\alpha$ we will fix from the following

\begin{lemma}\label{Mestimate}
For the function $f$ defined in \eqref{split} we have 
$$\|f'\|_\infty \leq e^{\alpha \lambda n}$$ 
with $\alpha := 2s + 2\eps$.
\end{lemma}

\begin{proof}
Differentiating term-wise, we obtain
$$|f'(x)| \leq 2\pi u  \sum_{\a \in \cR_n} w_\a(x) |T_\a'(x)|.$$
For $\a \in \cR_n$, we have by Lemma \ref{lma:regular}(2) that $|T_\a'(x)| \leq e^{(-\lambda+\eps)n}$. Now inputting $u < e^{(1+2s)\lambda(n+1)}$ from \eqref{choiceofn}, the Gibbs property $w_\a(x) \leq C\mu(I_\a)$ and $\lambda > 1$ we obtain
$$|f'(x)| \leq 2C\pi  u e^{(-\lambda+\eps)n}\sum_{\a \in \N}\mu(I_\a) \leq 2C\pi e^{(1+2s)\lambda }e^{(2s\lambda+\eps\lambda)n} \leq e^{\alpha \lambda n}$$
since by the choice \eqref{n04} of $n_0$ we have $2C\pi e^{(1+2s)\lambda} \leq e^{\eps \lambda n_0}$ (recall $n \geq n_0$).
\end{proof}

\subsection{Bounding the $L^2$ norm} 
\label{subsec:L2}

Thanks to Lemma \ref{lma:largesieve}, we are now reduced to just bound the $L^2$ norm of $f$. We dedicate this section to prove the following
\begin{proposition}\label{m2estimate}
For the function $f$ defined in \eqref{split} we have
\begin{equation}\label{neededm2}
\|f\|_2^2 \ll u^{-1/2} e^{(\frac{\lambda}{2}+15\lambda\eps)n}+ e^{(-s\lambda+3\lambda\eps)n}
\end{equation}
\end{proposition}

This is the part where we use heavily the nonlinear nature of the Gauss map and much of the number theoretical tools in continued fractions appear. This is also the part where we adapt the mirroring argument presented by Queff\'elec and Ramar\'e, which allows us to only assume that the Hausdorff dimension of $\mu$ is at least $1/2$ in contrast to Kaufman's bound $2/3$. The main idea of bounding the $L^2$ norm is to study the distribution of the differences $T_\a(x)-T_\b(x)$ for different $n$-regular words $\a$ and $\b$ and exploit \textit{stationary phase} type inequalities (see Section 6 of \cite{Wolff2003}) to bound oscillative Lebesgue integrals in the terms of derivatives of the oscillations that are under control for words in $\cR_n$. The key lemma is the following modification of the lemma used by Kaufman in \cite{Kaufman1980}.

\begin{lemma}[Stationary phase]
\label{lma:kaufmansineqs} Let $\psi : [0,1] \to \R$ be a $C^2$ map such that the derivative has the form
$$\psi'(x) = \phi(x) \cdot (\alpha_1 x + \alpha_2)$$
for some $C^1$ map $\phi : [0,1] \to \R$ with $a \leq |\phi| \leq b$ and $|\phi'| \leq b$, where $b > a > 0$.
\begin{itemize}
\item[(1)] If $\alpha_1 \neq 0$, then
$$\Big|\int e^{2\pi i \psi(x)}\, dx\Big| \leq 6ba^{-3/2} |\alpha_1|^{-1/2}.$$
\item[(2)] If $|\alpha_1| \leq |\alpha_2| / 2$ and $\alpha_2 \neq 0$, then
$$\Big|\int e^{2\pi i \psi(x)}\, dx\Big| \leq 8|\alpha_2|^{-1} (a^{-1}+ba^{-2}).$$
\end{itemize}
\end{lemma}
\begin{proof}
The proofs were given in Section 4 of \cite{Kaufman1980} and they follow immediately from partial integration since we know the growth bounds for $\psi$. However, the inequality presented in \cite{Kaufman1980} for the case (2) has a slightly different form. There, precisely speaking, it was proved that for any $C^2$ map $\psi : [0,1] \to \R$ with $|\psi'| \geq r$ and $|\psi''| \leq R$ we have
$$\Big|\int e^{2\pi i \psi(x)}\, dx\Big| \leq r^{-1}+Rr^{-2}.$$
However, this yields immediately the case (2) since by the assumptions on $\phi$ and $|\alpha_1| \leq |\alpha_2|/2$ the derivatives
$$|\psi'(x)| = |\phi(x)||\alpha_1 x + \alpha_2| \geq a|\alpha_2|/2$$
and
$$|\psi''(x)| = |\phi'(x)|(|\alpha_1| +|\alpha_2|)+|\phi(x)||\alpha_1| \leq 2b|\alpha_2|,$$
which gives the claim with $r = a|\alpha_2|/2$ and $R = 2b|\alpha_2|$.
\end{proof}

Now, let us look at the proof of Proposition \ref{m2estimate}, and how we can reduce the setting to a situation where we can apply the stationary phase. Given $\a,\b \in \cR_n$, define the $C^2$ mapping $\psi_{\a,\b} : [0,1] \to \R$ at $x \in [0,1]$ by
$$\psi_{\a,\b}(x) = \xi(T_\a(x)-T_\b(x)).$$
After expanding the square, we obtain
$$|f(x)|^2 = \sum_{\a,\b \in \cR_n}  w_\a(x) w_\b(x) e^{2\pi i \psi_{\a,\b}(x)},$$
where we have that $w_\a(x) = e^{S_n\phi(T_\a(x))}$. Writing
$$I(\a,\b):=\left|\int \exp(2\pi i \psi_{\a,\b}(x)) \,dx\right| \quad \text{and} \quad S(\a) := \sum_{\b \in \cR_n} \mu(I_\b) I(\a,\b),$$
the Gibbs property of $\mu$ yields
\begin{align}\label{thebound}\int |f(x)|^2 \, dx \ll  \sum_{\a,\b \in \cR_n}\mu(I_\a)\mu(I_\b)I(\a,\b) = \sum_{\a \in \cR_n}\mu(I_\a) S(\a).\end{align}
Thus we end up estimating the integrals $I(\a,\b)$ for possible choices of $\a$ and $\b$. For this purpose, we denote the difference of the $n$th continuants by
$$d_n(\mathbf{a},\mathbf{b}):=|q_n(\mathbf{a})-q_n(\mathbf{b})|.$$
Similarly, define $d_{n-1}$ using the $(n-1)$th continuants of $\a$ and $\b$. Notice that these differences attain only integer values. Then with the stationary phase we obtain a relationship between $I(\a,\b)$ and the differences $d_n$ and $d_{n-1}$ in the following key lemma:

\begin{lemma}\label{lma:keybounds}
Let $\a,\b\in\mathcal{R}_n$. If $d_{n-1}(\a,\b) \neq 0$, then
\begin{equation}\label{eq1}
I(\a,\b)\ll\frac{u^{-1/2}e^{(3\lambda/4 + 7\eps)n}}{d_{n-1}(\a,\b)^{1/2}}.
\end{equation}
and if $d_{n-1}(\a,\b) \leq d_n(\a,\b)/2$ and $d_n(\a,\b) \neq 0$ we have
\begin{equation}\label{eq2}
I(\a,\b)\ll\frac{u^{-1/2}e^{(3\lambda/4 + 7\eps)n}}{d_{n}(\a,\b)^{1/2}}.
\end{equation}
\end{lemma}

\begin{proof}
Fix $\a,\b \in \cR_n$ and write $\psi = \psi_{\a,\b}$ as defined above. Recall that the derivative of $T_\a$ has the form
$$T_\a'(x) = \frac{(-1)^n}{(q_{n-1}(\a)x+q_n(\a))^2}.$$
so writing $\alpha_1 := q_{n-1}(\a)-q_{n-1}(\b)$, $\alpha_2 := q_{n}(\a)- q_{n}(\b)$ and
\begin{align}\label{phi}\phi(x) := (-1)^n \xi \cdot \frac{(q_{n-1}(\a)+q_{n-1}(\b)) x + q_{n}(\a)+q_{n}(\b) }{(q_{n-1}(\a) x + q_{n}(\a))^2(q_{n-1}(\b) x + q_{n}(\b))^2}\end{align}
we obtain $\psi'(x) = \phi(x) \cdot (\alpha_1 x + \alpha_2)$. Thus we will apply Lemma \ref{lma:kaufmansineqs} using bounds for $|\phi|$ and $|\phi'|$. Since we have $q_{n-1}(\a) \leq q_n(\a)$ and
$$e^{(\lambda/2-\eps/2)n}\leq q_n(\a)\leq 4e^{(\lambda/2+\eps/2)n}$$ 
and the same bounds for the continuant $q_n(\b)$, we obtain the bound
$$a:= \tfrac{1}{8}ue^{(-3\lambda-5\epsilon)n/2} \leq  |\phi(x)| \leq 16ue^{(-3\lambda+5\epsilon)n/2} \leq 40ue^{(-3\lambda+5\eps)n/2} := b.$$
Differentiating again \eqref{phi}, we obtain as well
$$|\phi'(x)| \leq b.$$
Note that here we did the crude bound for $|\phi(x)|$ above as we wanted the same bound with $b$ as for $\phi'$ to use Lemma \ref{lma:kaufmansineqs}. Then $0 < a < b$ and moreover
$$ba^{-3/2} \ll u^{-1/2}e^{(3\lambda + 25\eps)n/4}.$$
Recall that by the choice \eqref{choiceofn} of $n$, we have $e^{(1+2s)\lambda n} \leq u$. Moreover, $\lambda > 1$ and $s > 1/2$ so
$$u^{-1}e^{(3\lambda + 25\eps)n/2} \leq e^{(\frac{1}{2}-2s+13\eps)\lambda n} \leq e^{(-\frac{1}{2}+13\eps)\lambda n} < 1$$
since, recall we chose initially $\eps < 1/26$ in \eqref{epschoice}. Thus by taking square roots, we have the same bounds for $ba^{-2}$ and $a^{-1}$ as well:
$$ba^{-2} \ll u^{-1}e^{(3\lambda + 15\eps)n/2} \leq u^{-1/2}e^{(3\lambda + 25\eps)n/4} < u^{-1/2}e^{(3\lambda/4 + 7\eps)n}$$
and
$$a^{-1} \ll u^{-1}e^{(3\lambda+5\epsilon)n/2} \leq u^{-1/2}e^{(3\lambda+25\epsilon)n/4} < u^{-1/2}e^{(3\lambda/4 + 7\eps)n}.$$
An application of the case (1) of Lemma \ref{lma:kaufmansineqs} with $\alpha_1 = q_{n-1}(\a)-q_{n-1}(\b)$ and $\alpha_2 = q_{n}(\a)- q_{n}(\b)$ then yields in the case $d_{n-1}(\a,\b) \neq 0$ that
$$I(\a,\b)\leq 6ba^{-3/2} |\alpha_1|^{-1/2} = 6ba^{-3/2}(d_{n-1}(\a,\b))^{-1/2} \ll \frac{u^{-1/2}e^{(3\lambda/4 + 7\eps)n}}{d_{n-1}(\a,\b)^{1/2}}.$$
In the case when $d_{n-1}(\a,\b) \leq d_n(\a,\b)/2$ and $d_n(\a,\b) \neq 0$ we use the case (2) of Lemma \ref{lma:kaufmansineqs} to obtain
$$I(\a,\b) \leq  8|\alpha_2|^{-1} (a^{-1}+ba^{-2}) \ll \frac{u^{-1/2}e^{(3\lambda/4 + 7\eps)n}}{d_{n}(\a,\b)^{1/2}}$$
since $d_n(\a,\b)^{1/2} \leq d_n(\a,\b)$ as $d_n(\a,\b) \geq 1$.
\end{proof} 

Now as we noted in \eqref{thebound} to prove Proposition \ref{m2estimate} we are left to estimate the sum
$$S(\a) = \sum_{\b \in \cR_n} \mu(I_\b) I(\a,\b)$$ 
for a fixed $\a \in \cR_n$. Write
$$r_0 := \frac{1}{192} e^{(\frac{\lambda}{2}-2\eps) n},$$
which, by the choice \eqref{n04} of $n_0$, is larger than or equal to $1$. We decompose the sum $S(\a)$ termwise to four smaller sums $A_1$, $A_2$, $A_3$ and $A_4$ depending on the values of $d_{n-1}(\a,\b)$ and $d_n(\a,\b)$ with respect to $r_0$ and to each other; see their definitions in the Lemmas \ref{m2lemma1}, \ref{m2lemma2}, \ref{m2lemma3} and \ref{m2lemma4} below. Then in these lemmas we will apply Lemma \ref{lma:keybounds} to obtain bounds for the sums $A_1,\dots,A_4$ that are independent of $\a$ as follows:
$$A_1 \ll e^{(-s\lambda+3\lambda\eps)n} \quad \text{and} \quad A_2,A_3,A_4 \ll u^{-1/2} e^{(\frac{\lambda}{2}+15\lambda\eps)n}.$$
Thus if we apply the bound \eqref{thebound} we have for $\|f\|_2^2$ and use the fact that $\mu$ is a probability measure, we obtain
$$\int |f(x)|^2 \, dx \ll \sum_{\a \in \cR_n}\mu(I_\a) S(\a) \ll e^{(-s\lambda+3\lambda\eps)n}+u^{-1/2} e^{(\frac{\lambda}{2}+15\lambda\eps)n}$$
as claimed in Proposition \ref{m2estimate}. Therefore, we are just left to describe the decomposition to $A_1,\dots,A_4$ and prove the desired bounds for them.

Let us first we consider words in the summation $S(\a)$ for which $d_n$ and $d_{n-1}$ agree:

\begin{lemma}
\label{m2lemma1}
We have
\begin{align*}A_1 := \sum_{\b \in \cR_n : \,d_{n-1}(\a,\b) = d_n(\a,\b) = 0} \mu(I_\b)I(\a,\b) \ll e^{(-s\lambda+3\lambda\eps)n}.\end{align*}
\end{lemma}
\begin{proof}
In this case we notice that whenever $\b \in \cR_n$ with
$$d_n(\a,\b) = d_{n-1}(\a,\b) = 0,$$
then in fact we have $\b = \a$. Indeed, if $q_n(\a) = q_n(\b)$ and $q_{n-1}(\a) = q_{n-1}(\b)$ and recall
$$q_n(\a) p_{n-1}(\a)  - q_{n-1}(\a) p_n(\a) = (-1)^n = q_n(\b) p_{n-1}(\b)  - q_{n-1}(\b) p_n(\b)$$
by \eqref{relation}, so we obtain 
$$q_{n-1}(\a)(p_n(\a) - p_n(\b)) \equiv 0 \mod q_n(\a).$$
As $q_{n-1}(\a)$ and $q_n(\a)$ are coprime, we obtain $p_n(\a) \equiv p_n(\b) \mod q_n(\a)$. However,
$$p_n(\b)/q_n(\a) = p_n(\b)/q_n(\b) \leq 1,$$ 
so this is only possible if $p_n(\a) = p_n(\b)$ yielding $\b = \a$ whenever $d_n(\a,\b) = d_{n-1}(\a,\b) = 0$. Thus by Lemma \ref{lma:regular}(4) we can bound
\begin{align*}\sum_{\b : \,d_{n-1}(\a,\b) = d_n(\a,\b) = 0} \mu(I_\b)I(\a,\b) = \mu(I_\a) \ll e^{(-s\lambda+3\lambda\eps)n}.\end{align*}
\end{proof}

In the second case we have words for which $d_n$ or $d_{n-1}$ is bounded below by $r_0$:

\begin{lemma}
\label{m2lemma2}
We have
\begin{align*}A_2 := \sum_{\b \in \cR_n :\,\d_{n}(\a,\b) > r_0 \text{ \emph{or} } d_{n-1}(\a,\b) > r_0} \mu(I_\b)I(\a,\b) \ll u^{-1/2} e^{(\frac{\lambda}{2}+15\lambda\eps)n}.\end{align*}
\end{lemma}
\begin{proof}
First of all, using \eqref{eq2} of Lemma \ref{lma:keybounds} with Lemma \ref{lma:regular}(4) we obtain
\begin{align*}\sum_{\b : \,\stack{d_{n}(\a,\b) > r_0}{d_{n-1}(\a,\b) \leq d_n(\a,\b)/2}} \mu(I_\b)I(\a,\b) \ll  u^{-1/2} r_0^{-1/2}e^{(-s\lambda+3\lambda\eps)n}  e^{(3\lambda/4 + 7\eps)n} \ll u^{-1/2} e^{(\frac{\lambda}{2} -s\lambda+12\lambda\eps)n},\end{align*}
by the definition of $r_0 = \frac{1}{192} e^{\frac{n\lambda}{2}-2\eps n}$. Secondly, using \eqref{eq1} of Lemma \ref{lma:keybounds} yields the same bound for the rest
\begin{align*}\sum_{\b : \,\stack{d_{n}(\a,\b) > r_0}{d_{n-1}(\a,\b) > d_n(\a,\b)/2}} \mu(I_\b)I(\a,\b) &\leq \sum_{\b : \,d_{n-1}(\a,\b) > r_0/2} \mu(I_\b)I(\a,\b)  \\
&\ll u^{-1/2} (r_0/2)^{-1/2}e^{(-s\lambda+3\lambda\eps)n}  e^{(3\lambda/4 + 7\eps)n} \ll u^{-1/2} e^{(\frac{\lambda}{2} -s\lambda+12\lambda\eps)n}
\end{align*}
Similary, \eqref{eq1} of Lemma \ref{lma:keybounds} gives us
\begin{align*}\sum_{\b : \,d_{n-1}(\a,\b) > r_0} \mu(I_\b)I(\a,\b) \ll  u^{-1/2} r_0^{-1/2}e^{(-s\lambda+3\lambda\eps)n}  e^{(3\lambda/4 + 7\eps)n} \ll u^{-1/2} e^{(\frac{\lambda}{2} -s\lambda+12\lambda\eps)n} .\end{align*}
This number is always less than the desired bound $u^{-1/2} e^{(\frac{\lambda}{2}+15\lambda\eps)n}$. 
\end{proof}

Thus we are just left with the part of the sum $S(\a)$ where both $d_n$ and $d_{n-1}$ are bounded from above by $r_0$ and where they are not simultaneously $0$. In this part we have two sums $A_3$ and $A_4$ defined depending on whether $d_n \leq 2d_{n-1}$ or not, and bounds for them in Lemmas \ref{m2lemma3} and \ref{m2lemma4} below. In the next lemma we have $d_n \leq 2d_{n-1}$ which allows us to exlude the case $d_{n-1} = 0$ as then it would force $d_n = 0$, which is already handled in Lemma \ref{m2lemma1} above.

\begin{lemma}
\label{m2lemma3}
We have
\begin{align*}
A_3:= \sum_{\b \in \cR_n : \,\stack{d_{n}(\a,\b) \leq 2d_{n-1}(\a,\b)}{1 \leq d_{n-1}(\a,\b) \leq r_0, d_{n}(\a,\b) \leq r_0}}\mu(I_\b)I(\a,\b) \ll u^{-1/2} e^{(\frac{\lambda}{2}+15\lambda\eps)n}.
\end{align*}
\end{lemma}
\begin{proof}
For now on, write $m := \lceil \tfrac{n}{2} \rceil$. Given $j = 0,1,\dots,m$, let us write
$$r_j := \frac{1}{192} e^{(\frac{n}{2}-j)\lambda-2\eps n}.$$
For $j = 0$ this agrees with the definition of $r_0$ given above and $r_0 \geq \frac{1}{192} e^{(\frac{\lambda}{2}-2\eps) n_0} \geq 1$ by the choice \eqref{n04} of $n_0$ and because by the choice of $\eps$ in \eqref{epschoice} we have $\eps < 1/26 < \lambda/4$ by $\lambda > 1$. Note that $r_0 > r_1 > \dots > r_m$ and $r_m < 1$. For $j = 0,1,\dots,m-1$ define the annulus
$$\cA_j = \{\b \in \cR_n :  d_{n}(\a,\b) \leq 2r_{j} \text{ and } r_{j+1} \leq d_{n-1}(\a,\b) \leq r_j\}.$$
Now these annuli allow us to decompose the summation we consider. Since $\mu(I_\b) \ll e^{(-s\lambda+3\lambda\eps)n}$ by Lemma \ref{lma:regular}(4) we have using the bound \eqref{eq1} of Lemma \ref{lma:keybounds} that
\begin{align*}\sum_{\b : \,\stack{d_{n}(\a,\b) \leq 2d_{n-1}(\a,\b)}{1 \leq d_{n-1}(\a,\b) \leq r_0, d_{n}(\a,\b) \leq r_0}} \mu(I_\b)I(\a,\b) &\leq \sum_{j = 1}^{m-1} \sum_{\b \in \cA_j}\mu(I_\b)I(\a,\b) \\
&\ll u^{-1/2}e^{(3\lambda/4 + 7\eps)n} \sum_{j = 0}^{m-1} r_j^{-1/2} \sum_{\b \in \cA_j} e^{(-s\lambda+3\lambda\eps)n}\\
& = u^{-1/2}e^{(3\lambda/4 -s\lambda + 10\lambda\eps)n} \sum_{j = 0}^{m-1} r_j^{-1/2} |\cA_j|
\end{align*}
Thus our estimate reduces to estimating the cardinality of $\cA_j$. Fix $\b \in \cA_j$. Recall that the \textit{mirror images} of $\a$ and $\b$ are defined by
$$\a^\leftarrow = (a_n,a_{n-1},\dots,a_1) \quad \text{and} \quad \b^\leftarrow = (b_n,b_{n-1},\dots,b_1).$$
By the mirroring property \eqref{reverse} of continuants we obtain
$$d_{n-1}(\a,\b) = |p_n(\a^\leftarrow) - p_n(\b^\leftarrow)| \quad\text{and}\quad d_n(\a,\b) = d_n(\a^\leftarrow,\b^\leftarrow)$$
so $r_{j+1} \leq |p_{n}(\a^\leftarrow) - p_{n}(\b^\leftarrow)| \leq r_j$ and $d_n(\a^\leftarrow,\b^\leftarrow) \leq 2r_j$, which yields
$$\Big|\frac{p_n(\a^\leftarrow)}{q_n(\a^\leftarrow)} - \frac{p_n(\b^\leftarrow)}{q_n(\b^\leftarrow)}  \Big| \leq \frac{ |p_{n}(\a^\leftarrow) - p_{n}(\b^\leftarrow)|}{q_n(\a^\leftarrow)}+\frac{|q_{n}(\a^\leftarrow) - q_{n}(\b^\leftarrow)|p_n(\b^\leftarrow)}{q_n(\a^\leftarrow)q_n(\b^\leftarrow)} \leq \frac{3r_j}{q_n(\a)},$$
where we used the mirroring which gives $q_n(\a) = q_n(\a^{\leftarrow})$. 

By the quasi-independence Lemma \ref{lma:quasiindependence} we have
$$q_n(\a) \geq \tfrac{1}{2}q_{n-j}(a_1,\dots,a_{n-j})q_j(a_{n-j+1},\dots,a_n).$$
On the other hand, by the mirroring property, we have
$$q_j(a_{n-j+1},\dots,a_n) = q_j(a_n,\dots,a_{n-j+1})$$
so we can bound
$$\frac{q_{n-j}(a_1,\dots,a_{n-j})^2}{4q_n(\a)^2} \leq q_j(a_n,\dots,a_{n-j+1})^{-2}.$$
On the other hand, the definition $r_j = \frac{1}{192} e^{(\frac{n}{2}-j)\lambda-2\eps n}$ gives us
$$3r_j = \frac{1}{64} e^{(\frac{n}{2}-j)\lambda-2\eps n} < \frac{1}{64} \frac{e^{(\lambda-\eps)(n-j)}}{e^{(\lambda+\eps)n/2}} \leq  \frac{1}{64} \frac{ q_{n-j}(a_1,\dots,a_{n-j})^2}{q_n(\a)}$$
using the fact that $\lfloor n/2 \rfloor \leq n-j < n$ and Lemma \ref{lma:regular}(1). Thus we have shown that the annulus $\cA_j \subset \cB_j$, where
$$\cB_j := \Big\{\b \in \cR_n : \Big|\frac{p_n(\a^{\leftarrow})}{q_n(\a^\leftarrow)} - \frac{p_n(\b^\leftarrow)}{q_n(\b^\leftarrow)}  \Big| < \tfrac{1}{16}q_{j}(a_n,\ldots,a_{n-j+1})^{-2}\Big\}.$$
We wish to show that the cardinality of $\cB_j$ is at most $2|\cC_j|$, where
$$\cC_j := \{(b_1,\dots,b_{n-j+1}) : \b \in \cR_n\}.$$
Note that $\cC_j$ is precisely the collection of all generation $n-j+1$ subwords of elements in $\cR_n$. Since $\lfloor n/2\rfloor < n-j+1 \leq n$, we may use Lemma \ref{lma:regular} for the words in $\cC_j$. To prove this consider the interval $I = I_{a_n,\ldots,a_{n-j+1}}$. 
\begin{itemize}
\item[(a)] If $a_{n-j+1}\neq 1$, then $I$ is neighboured by the intervals $I_{a_n,\ldots,a_{n-j+1}-1}$ and $I_{a_n,\ldots,a_{n-j+1}+1}$ both of which must have diameter greater than $\tfrac{1}{16}q_{j}(a_n,\ldots,a_{n-j+1})^{-2}$. Indeed, by the recurrence relation for the continuants
\begin{align*}q_j(a_n,\ldots,a_{n-j+1}+1)& = q_j(a_n,\ldots,a_{n-j+1}) + q_{j-1}(a_n,\ldots,a_{n-j+1})\\
& \leq 2q_{j}(a_n,\ldots,a_{n-j})
\end{align*}
and
\begin{align*}q_j(a_n,\ldots,a_{n-j+1}-1) &= q_{j}(a_n,\ldots,a_{n-j}) - q_{j-1}(a_n,\ldots,a_{n-j+1})\\
& \leq q_{j}(a_n,\ldots,a_{n-j}),
\end{align*}
which thanks to Lemma \ref{lma:growth} gives the size bound $\tfrac{1}{16}q_{j}(a_n,\ldots,a_{n-j+1})^{-2}$ for the intervals from below. Thus if we require $\b\in \cB_j$ we would need $(a_n,\ldots,a_{n-j+2})=(b_n,\ldots,b_{n-j+2})$ as otherwise the point $p_n(\b^\leftarrow)/q_n(\b^\leftarrow)$ would be too far from the point $p_n(\a^\leftarrow)/q_n(\a^\leftarrow)$.
\item[(b)] If $a_{n-j+1}=1$, then depending on the sign of the derivative of $T_{a_n\dots,a_{n-j+2}}'$, the two neighbouring cylinders are the $j$th generation interval $I_{a_n,\ldots,a_{n-j+1}+1}$ and the $(j-1)$th generation interval $I_{a_n,\ldots,a_{n-j+2}+1}$ and both of these have diameter greater than $\tfrac{1}{16}q_{j}(a_n,\ldots,a_{n-j+1})^{-2}$ again by the recurrence relation for the continuants. Thus there are at most two choices for $(b_n,\ldots,b_{n-j+2})$.
\end{itemize}
The cases (a) and (b) together prove that a given $\b \in \cB_j$ has at most two possibilities for the segment $b_n,\ldots,b_{n-j+2}$. Moreover, as $\b \in \cR_n$, the initial segment $(b_1,\dots,b_{n-j+1}) \in \cC_j$. Hence $|\cB_j| \leq 2|\cC_j|$ as we wanted. 

Now, we can just bound the cardinality of $\cC_j$. Given $(b_1,\dots,b_{b-j+1}) \in \cC_j$ we have by Lemma \ref{lma:regular}(4) that
$$\mu(I_{b_1,\ldots,b_{n-j+1}})\geq C^{-1} e^{-(n-j+1)(\lambda s+3\lambda\epsilon)}.$$
Since $\mu$ is a probability measure this gives us a cardinality bound
$$|\cC_j|\leq Ce^{(n-j+1)(\lambda s+3\lambda\epsilon)} \leq C e^{(s\lambda  + 3\lambda\eps) n}e^{-s\lambda j}.$$
Therefore, we can bound the cardinality of $\cA_j$ as follows
$$|\cA_j| \leq 2|\cC_j| \ll e^{(s\lambda  + 3\lambda\eps) n}e^{-s\lambda j}.$$
Now recall that we are interested in the sum $ \sum_{j = 0}^{m-1} r_j^{-1/2} |\cA_j| $. Inputting the definition of $r_j = \frac{1}{192} e^{(\frac{n}{2}-j)\lambda-2\eps n}$, we see that
$$r_j^{-1/2} \ll e^{(-\frac{\lambda}{4}+2\lambda\eps)n} e^{\frac{\lambda}{2}j}$$
using a crude bound $\sqrt{\lambda} \leq \lambda$ as $\lambda > 1$. Thus the sum
\begin{align}\label{sum2} \sum_{j = 0}^{m-1} r_j^{-1/2} |\cA_j|  \ll e^{(s\lambda  + 3\lambda\eps) n} e^{(-\frac{\lambda}{4}+2\lambda\eps)n}  \sum_{j = 0}^{m-1}e^{(\frac{1}{2}-s)\lambda j} \ll e^{(s\lambda  + 3\lambda\eps) n} e^{(-\frac{\lambda}{4}+2\lambda\eps)n} = e^{(-\frac{\lambda}{4}+s\lambda+5\lambda\eps)n}\end{align}
where we used $s > 1/2$ so that the geometric sum is dominated by a constant. Inputting this to what we obtained previously, we have our desired claim
\begin{align*}
\sum_{\b : \,\stack{d_{n}(\a,\b) \leq 2d_{n-1}(\a,\b)}{1 \leq d_{n-1}(\a,\b) \leq r_0, d_{n}(\a,\b) \leq r_0}} \mu(I_\b)I(\a,\b) \ll u^{-1/2}e^{(3\lambda/4 -s\lambda + 10\lambda\eps)n} \sum_{j = 0}^{m-1} r_j^{-1/2} |\cA_j| \ll u^{-1/2} e^{(\frac{\lambda}{2}+15\lambda\eps)n}.
\end{align*}
\end{proof}

Finally we are left with the case $d_n > 2d_{n-1}$. In this case we can exlude those words with $d_{n} = 0$ as then it would force $d_{n-1} < 0$, which is absurd as the differences are non-negative.

\begin{lemma}
\label{m2lemma4}
We have
$$A_4 := \sum_{\b \in \cR_n : \,\stack{d_{n}(\a,\b) > 2d_{n-1}(\a,\b)}{d_{n-1}(\a,\b) \leq r_0, 1 \leq d_{n}(\a,\b) \leq r_0}} \mu(I_\b)I(\a,\b) \ll u^{-1/2} e^{(\frac{\lambda}{2}+15\lambda\eps)n}.$$
\end{lemma}
\begin{proof}
In this case we can use the part \eqref{eq2} of Lemma \ref{lma:keybounds}, which gives
$$\sum_{\b : \,\stack{d_{n}(\a,\b) > 2d_{n-1}(\a,\b)}{d_{n-1}(\a,\b) \leq r_0, 1 \leq d_{n}(\a,\b) \leq r_0}} \mu(I_\b)I(\a,\b) \ll u^{-1/2} e^{(3\lambda/4-s\lambda + 10\lambda\eps)n}\sum_{\b : \,\stack{d_{n}(\a,\b) > 2d_{n-1}(\a,\b)}{d_{n-1}(\a,\b) \leq r_0, 1 \leq d_{n}(\a,\b) \leq r_0}} d_n(\a,\b)^{-1/2}$$
again by Lemma \ref{lma:regular}(4). Now consider the summation on the right-hand side. We proceed similarly as in the proof of Lemma \ref{m2lemma3}: now we split the summation to the annuli determined by the conditions $d_{n-1}(\a,\b) \leq r_j/2$ and $r_{j+1} \leq d_{n}(\a,\b) \leq r_j$ with the same radii $r_j$. By the mirroring property \eqref{reverse} of continuants these conditions yield
$$ |p_n(\a^\leftarrow) - p_n(\b^\leftarrow)| \leq r_j/2 \quad\text{and}\quad |q_{n}(\a^\leftarrow) - q_{n}(\b^\leftarrow)| \leq r_j$$
so we can make a same bound as in the proof of Lemma \ref{m2lemma3} for the difference:
$$\Big|\frac{p_n(\a^\leftarrow)}{q_n(\a^\leftarrow)} - \frac{p_n(\b^\leftarrow)}{q_n(\b^\leftarrow)}  \Big| \leq \frac{ |p_{n}(\a^\leftarrow) - p_{n}(\b^\leftarrow)|}{q_n(\a^\leftarrow)}+\frac{|q_{n}(\a^\leftarrow) - q_{n}(\b^\leftarrow)|p_n(\b^\leftarrow)}{q_n(\a^\leftarrow)q_n(\b^\leftarrow)} < \frac{3r_j}{q_n(\a)}.$$
Thus the proof after this is exactly as in Lemma \ref{m2lemma3} since this allows us to reduce the estimation to a cardinality bound for $\cB_j$ and summation over $j$. Hence we obtain the same bound up to a constant 
$$\sum_{\b : \,\stack{d_{n}(\a,\b) > 2d_{n-1}(\a,\b)}{d_{n-1}(\a,\b) \leq r_0, 1 \leq d_{n}(\a,\b) \leq r_0}} \mu(I_\b)I(\a,\b) \ll u^{-1/2} e^{(\frac{\lambda}{2}+15\lambda\eps)n}.$$
\end{proof}

\subsection{Completion of the proof of the main theorem}
\label{subsec:completion}

Recall that in Section \ref{subsec:split} we deduced that the proof of Theorem \ref{thm:main} is complete if we are able to prove that the main integral bound \eqref{mainbound} holds for the regular part, that is,
\begin{align*}\int |f(x)| \, d \mu(x) \ll u^{-\eta_s + 38 \eps} + u^{-\frac{\delta}{12\lambda}},\end{align*}
where the exponent
$$\eta_s = \frac{2s^2-s}{(4-s)(1+2s)}$$
and $\delta = \delta(\eps/2) > 0$ is the constant from large deviations with $\eps/2$. To verify this, write
$$\alpha := 2s + 2\eps \quad \text{and} \quad \beta := (1+2s)\eta_s.$$
Then $\gamma := \alpha + \beta \geq 1$ since $s > 1/2$ and crudely bounding also $\gamma \leq 7$. By Lemma \ref{Mestimate} we have $\|f'\|_\infty \leq e^{\alpha \lambda n}$. Thus the function $f$ satisfies the assumptions of Lemma \ref{lma:largesieve}. This gives us
$$\int |f(x)| \, d \mu(x) \ll e^{(-\beta \lambda+\eps\gamma ) n} +  \|f\|_2^2 e^{((1-s)\alpha \lambda + (3-s)\beta \lambda  +3\lambda\eps \gamma) n} + e^{-\delta\gamma n/4} =: t_1+t_2+t_3.$$
Recall that we chose in \eqref{choiceofn} the number $n \geq n_0$ such that
\begin{align}\label{ubound}
e^{(1+2s)\lambda n} \leq u < e^{(1+2s)\lambda(n+1)}.
\end{align}
By the definition of $\beta$ as $\lambda > 1$ and $\gamma \leq 7$ we obtain
$$t_1 = e^{-(1+2s)\eta_s\lambda n+\gamma \eps n} \ll u^{-\eta_s + \frac{\gamma}{(1+2s)\lambda} \eps} < u^{-\eta_s + 7 \eps}$$
and similarly as $\gamma > 1$ and $s \leq 1$ we have
$$t_3 \ll u^{-\frac{\gamma }{4(1+2s)\lambda}\delta} < u^{-\frac{\delta}{12\lambda}} .$$ 
Thus we are just left with bounding $t_2$ with respect to $u$. 

Recall the $L^2$ bound from Proposition \ref{m2estimate}:
$$\|f\|_2^2 \ll u^{-1/2}e^{(\lambda/2+15\lambda\eps)n}+ e^{(-s\lambda+3\lambda\eps)n}.$$
Inputting this to the definition of $t_2$ yields
\begin{align*}t_2 &\ll u^{-1/2}e^{(\lambda/2+15\lambda\eps)n} e^{((1-s)\alpha \lambda + (3-s)\beta \lambda  +3\lambda\eps \gamma) n} + e^{(-s\lambda+3\lambda\eps)n} e^{((1-s)\alpha \lambda + (3-s)\beta \lambda  +3\lambda\eps \gamma) n}.
\end{align*}
Now define
$$\rho_s := \frac{1}{2} - \frac{1/2 + 2s(1-s) + (3-s)(1+2s)\eta_s}{1+2s} = \frac{s - 2s(1-s) - (3-s)(1+2s)\eta_s}{1+2s}.$$
If we plug-in the definition of $\alpha = 2s + 2\eps$ and $\beta = (1+2s)\eta_s$ to the first term bounding $t_2$ and use \eqref{ubound}, we obtain a bound
\begin{align*}u^{-1/2}e^{(\lambda/2+15\lambda\eps)n} e^{((1-s)\alpha \lambda + (3-s)\beta \lambda  +3\lambda\eps \gamma) n} &= u^{-1/2}e^{(\lambda/2 + 2s(1-s) \lambda + (3-s)(1+2s)\eta_s \lambda +(3\gamma+17)\lambda\eps) n} \\
&\ll u^{-\rho_s+\frac{3\gamma+17}{1+2s} \eps} \\
& \ll u^{-\rho_s + 38\eps}
\end{align*}
using $\gamma \leq 7$. Moreover, for the second term bounding $t_2$ we obtain similarly
$$e^{(-s\lambda + 2s(1-s) \lambda + (3-s)(1+2s)\eta_s \lambda +(3\gamma+5)\lambda\eps) n} \ll u^{-\rho_s+\frac{3\gamma+5}{1+2s}  \eps} \ll u^{-\rho_s+26\eps} .$$
On the other hand, by the definition of $\eta_s = \frac{2s^2-s}{(4-s)(1+2s)}$, the exponent
$$\rho_s =  \frac{(4-s)[s - 2s(1-s) - (3-s)(1+2s)\eta_s]}{(4-s)(1+2s)} = \frac{8s^3 +10s^2 - s}{(4-s)(1+2s)} > \eta_s.$$
Thus the polynomial decay coming from the term $t_1$ involving $\eta_s$ dominates the decay we obtain from $t_2$ since $\eta_s - 38\eps > 0$ by the choice of $\eps$ in \eqref{epschoice}. This finishes the final claim and the proof of Theorem \ref{thm:main}.

\subsection{The case of Hausdorff measure}
\label{subsec:hausdorff}

To finish the section, let us make some remarks on the proof of Corollary \ref{hm} and see how the proof of the main Theorem \ref{thm:main} can be modified to have the Rajchman property for the Hausdorff measure $\nu = \cH^s|_{B(\cA)}$ for a finite $\cA \subset \N$. The crucial observation is that even though $\nu$ is not $T$ invariant, it is an $s$-\textit{conformal measure}, that is, a fixed point for the dual operator $\cL_{\phi}^*$ with potential $\phi=-s\log |T'|$; see for example \cite[Theorem 6]{Hofbauer1992}. In Section \ref{subsec:split}, when we apply $\cL_\phi$ to the oscillation $x \mapsto e^{-2\pi i \xi x}$, we only need the measure to be a fixed point for $\cL_\phi^*$. The $T$ invariance is only needed in the deduction of the large deviation bounds, but by \cite[Theorem 5.3]{Falconer1997} we know that $\nu$ is equivalent to a $T$ invariant measure $\mu$ on $B(\cA)$ with uniformly positive and finite densities. This allows us to use the same large deviation bounds or measure comparisons to the lengths $|I_\a|$ for $\nu$, which we already established for $\mu$ up to a fixed multiplicative error constant, which does not affect the decay rate of the Fourier transform.

\section{Prospects}
\label{sec:prospects}

\subsection{Dimension and decay}

An immediate question arising from our main result is that why is there a dimension assumption $\dim \mu > 1/2$ for the Gibbs measures $\mu$? The result by Hochman-Shmerkin \cite{HochmanShmerkin2013} on finding normal numbers in the supports of Gibbs measures does not have dimension restrictions, so it could be possible that the dimension requirement is just an artifact of the method we use. The dimension assumption is important in the summation of the geometric series in the $L^2$ estimations on the lemmas Lemma \ref{m2lemma3} and \ref{m2lemma4} of the proof of Proposition \ref{m2estimate} in Section \ref{subsec:L2}. In particular, in the estimate \eqref{sum2} the important term $r_0$ may not dominate if $s$ is less than $1/2$.

Another direction is to estimate the decay rate we have for the Fourier transform. It is an open problem whether the set of badly approximable numbers is a Salem set. As we noted in the introduction, the Hausdorff dimension of this set is $1$, so we would need to find Rajchman measures on badly approximable numbers with $|\widehat{\mu}(\xi)| = O(|\xi|^{-\eta})$ with $\eta$ arbitrarily close to $1/2$. The rates we obtain for the Gibbs measures when the Hausdorff dimension $s$ is large are not close to $1/2$. The proofs in this paper are designed to give polynomial decay but we have not considered how to optimise this decay and it is probable that new methods would need to be developed to do this.

\subsection{Markov maps} 

It is a natural question whether the results here could be extended to dynamical systems other than $([0,1] \setminus \Q,T)$. A possible problem could be to consider general expanding Markov maps on the circle $\mathbb{T}$ and the Gibbs measures related to these maps. Nonlinearity of the system seems to be crucial in the proof and in fact it is straightforward to see that the results cannot hold for the maps $x\mapsto nx\mod 1$. However, even if we assume nonlinearity, another obstacle for extending the results are some of the number theoretical properties of the Gauss map we use. The mirroring property which basically guarantees that the interval $I_\a$ has a comparable length to $I_{\a^\leftarrow}$ we used in bounding the $L^2$ norm of the regular part $f$ in Section \ref{subsec:L2} is rarely available to general Markov maps. This can probably be overcome by slightly stregthening the assumption on the dimension of $\mu$ since this property is only needed to weaken this assumption. However in part 1 of the proof of Proposition \ref{m2estimate} we estimate the set for which $d_{n-1}(\mathbf{a},\mathbf{b})=0$ and $d_n(\mathbf{a},\mathbf{b})=0$. This makes important use of properties of the continuants of the continued fraction expansion and so this part of the proof would need to be overhauled to extend the result to more general Markov maps. 
\section{Acknowledgements}
The authors are grateful to  Alan Haynes, Andrew Ferguson and all the other members of the \textit{``Dynamical Systems and Diophantine Approximation''} reading group in the University of Bristol during Spring 2013 where this collaboration was initiated. TS also thanks Sanju Velani and all the participants of the minicourse \textit{``An Invitation to Geometric and Diophantine Fourier Analysis''} in the University of York during December 2013 for many comments and suggestions on this paper. We thank Tuomas Orponen, Micha{\l} Rams and Mike Todd for useful comments on a preliminary version of this work, Tomas Persson for pointing out the application to Salem's problem, and the anonymous referee for comments and suggestions.

\bibliographystyle{abbrv} 
\bibliography{FourierGibbsFinal}

\end{document}